\documentclass[11pt, reqno]{amsart}

\usepackage{amssymb}
\usepackage{amscd}
\usepackage{amsfonts}
\usepackage{version}
\usepackage{float}
\usepackage{color}
\usepackage[
backend=biber,
style=alphabetic,
]{biblatex}
\usepackage{bookmark}
\usepackage{enumitem}

\usepackage{amssymb}
\usepackage[top=1in, bottom=1in, left=1in, right=1in]{geometry}

\newtheorem{theorem}{Theorem}[section]
\newtheorem{lemma}{Lemma}[section]

\theoremstyle{definition}
\newtheorem{definition}{Definition}[section]
\newtheorem{example}{Example}[section]

\newtheorem{corollary}{Corollary}[section]

\theoremstyle{remark}
\newtheorem{remark}{Remark}[section]

\numberwithin{equation}{section}

\newcommand{\Mod}[1]{\ (\mathrm{mod}\ #1)}

\renewcommand{\Re}{\mathrm{Re}}
\renewcommand{\Im}{\mathrm{Im}}

\renewcommand{\leq}{\leqslant}
\renewcommand{\geq}{\geqslant}

\begin{document}

\title{Dirichlet law for factorization of integers, polynomials and permutations}


\author{}
\address{}
\curraddr{}
\email{}
\thanks{}

\author{Sun-Kai Leung}
\address{D\'epartement de math\'ematiques et de statistique\\
Universit\'e de Montr\'eal\\
CP 6128 succ. Centre-Ville\\
Montr\'eal, QC H3C 3J7\\
Canada}
\curraddr{}
\email{sun.kai.leung@umontreal.ca}
\thanks{}


\date{}

\dedicatory{}

\keywords{}

\begin{abstract}
Let $k \geq 2$ be an integer.
We prove that factorization of integers into $k$ parts follows the Dirichlet distribution $\mathrm{Dir}\left(\frac{1}{k},\ldots,\frac{1}{k}\right)$ by multidimensional contour integration, thereby generalizing the Deshouillers–Dress–Tenenbaum (DDT) arcsine law on divisors where $k=2$. The same holds for factorization of polynomials or permutations. Dirichlet distribution with arbitrary parameters can be modelled similarly.
\end{abstract}

\maketitle

\section{Introduction}

Given an integer $n \geq 1$, it is natural to study the distribution of its divisors over the interval ${[1,n]}$ (in logarithmic scale). Let $d$ be a random integer chosen uniformly from the divisors of $n.$ Then $D_n:=\frac{\log d}{\log n}$ is a random variable taking values in $[0,1].$ While one can show that
the sequence of random variables $\{D_n\}_{n=1}^{\infty}$ does not converge in distribution, Deshouillers, Dress and Tenenbaum \cite{tenenbaum1980lois} proved the 
mean of the corresponding distribution functions converges to that of the arcsine law. More precisely, uniformly for $u \in [0,1],$ we have
\begin{align*}
\frac{1}{x}\sum_{n \leq x}\mathbb{P}\left(D_n \leq u\right)=
\frac{2}{\pi} \arcsin{\sqrt{u}}
+O\left(\frac{1}{\sqrt{\log x}}\right),
\end{align*}
where 
\begin{align*}
\mathbb{P}\left(D_n \leq u\right):=
\frac{1}{\tau(n)}\sum_{\substack{d | n \\ d \leq n^{u}}}1
\end{align*}
is the distribution function of $D_n$ and the 
error term here is optimal (see also \cite[Chapter~6.2]{tenenbaum2015introduction}). 

Recently, Nyandwi and Smati \cite{nyandwi2014distribution} studied the  distribution of pairs of divisors of a given integer on average.
Similarly, they also proved the 
mean of the corresponding distribution functions converges to that of the beta two-dimensional law uniformly together with the optimal rate of convergence.

Our main goal here is to generalize their work to higher dimensions, which they claim is very technical following the usual approach (see \cite[p. 2]{nyandwi2018distribution}). Fix $k \geq 2.$ Given an integer $n \geq 1,$ let $(d_1,\ldots,d_k)$ be a random $k$-tuple chosen uniformly from the set of all possible factorization $\{(m_1,\ldots,m_k) \in \mathbb{N}^k \, : \, n=m_1\cdots m_k\}.$ Then $\boldsymbol{D}_n=(D_n^{(1)},\ldots,D_n^{(k)}):=\left(\frac{\log d_1}{\log n},\ldots,\frac{\log d_k}{\log n}\right)$ is a multivariate random variable taking values in $[0,1]^k.$ Similarly, we are interested in the mean
\begin{align*}
\frac{1}{x}\sum_{n \leq x}\mathbb{P}
\left(D_n^{(1)}\leq u_1,\ldots,D_n^{(k-1)}\leq u_{k-1}\right),
\end{align*}
where
\begin{align*}
\mathbb{P}\left(D_n^{(1)}\leq u_1,\ldots,D_n^{(k-1)}\leq u_{k-1}\right):=
\frac{1}{\tau_k(n)}\underset{d_1\cdots d_{k-1}|n}{\sum_{d_1 \leq n^{u_1}}\cdots\sum_{d_{k-1}\leq n^{u_{k-1}}}}1
\end{align*}
is the distribution function of $\boldsymbol{D}_n.$

Note that since $n=d_1\cdots d_k,$ the multivariate random variable $\boldsymbol{D}_n$ must satisfies 
\begin{align*}
1=D_n^{(1)}+\cdots+D_n^{(k)},
\end{align*}
and so it actually takes values in the $(k-1)$-dimensional probability simplex. We now turn to the Dirichlet distribution, which
is the most natural candidate of modeling such distribution.  
\begin{definition}
Let $k \geq 2.$ The Dirichlet distribution of dimension $k$ with parameters $\alpha_1,\ldots,\alpha_k>0$ is denoted by
$\mathrm{Dir}\left(\alpha_1,\ldots,\alpha_k\right), $ which is 
defined on the $(k-1)$-dimensional probability simplex
\begin{align*}
\Delta^{k-1}:=\{(t_1,\ldots,t_k) \in [0,1]^k \, : \, t_1+\cdots+t_k=1\}
\end{align*}
having density 
\begin{align*} 
f_{\boldsymbol{\alpha}}(t_1,\ldots,t_k):=\frac{\Gamma \left(\sum_{i=1}^k \alpha_i \right)}{\prod_{i=1}^{k}\Gamma(\alpha_i)}\prod_{i=1}^{k}t_i^{\alpha_i-1}.
\end{align*}
\end{definition}
For instance, when $k=2,$ Dirichlet distribution reduces to beta distribution  $\mathrm{Beta}\left(\alpha, \beta\right)$ with parameters $\alpha, \beta$.
In particular, $\mathrm{Beta}\left(\frac{1}{2}, \frac{1}{2}\right)$ is the arcsine distribution.

As we will see, factorization of integers into $k$ parts follows the Dirichlet distribution $\mathrm{Dir}\left(\frac{1}{k},\ldots,\frac{1}{k}\right).$
Since for each $i$ the parameter $\alpha_i=\frac{1}{k}$ is less than $1,$ the density $f_{\boldsymbol{\alpha}}(t_1,\ldots,t_k)$ blows up most rapidly at the $k$ vertices of the probability simplex. Therefore, our intuition that a typical factorization of integers into $k$ parts consists of one large factor and $k-1$ small factors is justified quantitatively.

By definition, for $u_1,\ldots,u_{k-1} \geq 0$ satisfying $u_1+\cdots+u_{k-1}\leq 1,$ the distribution function of $\mathrm{Dir}\left(\alpha_1,\ldots,\alpha_k \right)$ is given by
\begin{align*}
F_{\boldsymbol{\alpha}}(u_1,\ldots,u_{k-1}):=
\int_0^{u_1}\cdots\int_0^{u_{k-1}}f_{\boldsymbol{\alpha}}(t_1,\ldots,t_{k-1}, 1-t_1-\cdots -t_{k-1})
dt_1\cdots dt_{k-1}.
\end{align*}
From now on until Section \ref{section5}, we shall fix $\boldsymbol{\alpha}=\left(\frac{1}{k},\ldots,\frac{1}{k}\right)$ and omit the subscript.

The main results are stated as follows.

\begin{theorem} \label{thm:dirichlet} Let $k\geq 2$ be a fixed integer. Then uniformly for $x \geq 2$ and $u_1,\ldots,u_{k-1} \geq 0$ satisfying $u_1+\cdots+u_{k-1}\leq 1,$
we have
\begin{align} \label{fundamental}
\frac{1}{x}\sum_{n \leq x}\frac{1}{\tau_k(n)}\underset{d_1\cdots d_{k-1} \mid n}{\sum_{d_1 \leq n^{u_1}}\cdots
\sum_{d_{k-1} \leq n^{u_{k-1}}}}1
=F(u_1,\ldots,u_{k-1}) 
+O\left(\frac{1}{(\log x)^{\frac{1}{k}}}\right).
\end{align}
\end{theorem}

The error term here is optimal if full uniformity in $u_1,\ldots,u_{k-1}$ is required. Indeed, if we choose
$u_1=\cdots=u_{k-2}=\frac{1}{k}, u_{k-1}=0,$ then one can show that the left-hand side of (\ref{fundamental}) is of order $(\log x)^{-\frac{1}{k}}$ using \cite[Th\'eor\`eme~ T]{tenenbaum1980lois} followed by partial summation.

\begin{remark}
Instead of using the logarithmic scale, one may also study localized factorization of integers, say for instance the quantity
\begin{align*}
H^{k}(x, \boldsymbol{y}, \boldsymbol{z}):=
 \left|\left\{ n \leq x \ : 
 \begin{array}{c}
  \text{there exists $(d_1,\ldots,d_{k-1}) \in \mathbb{N}^{k-1}$ such that } \\
  \text{$d_1\cdots d_{k-1} | n$ and $y_i<d_i \leq z_i$ for $i=1,\ldots,k-1$}
 \end{array}
 \right\}\right|,
\end{align*}
which was discussed in \cite{koukoulopoulos2010localized}.
\end{remark}

Note that Theorem \ref{thm:dirichlet} implies that for any axis-parallel rectangle $R \subseteq \Delta^{k-1},$ we have
\begin{align*}
\frac{1}{x}\sum_{n \leq x}\mathbb{P}
\left( \boldsymbol{D}_n \in R \right)
=\int_{R} dF
+O\left(\frac{1}{(\log x)^{\frac{1}{k}}}\right).
\end{align*}
Since every Borel subset of the simplex can be approximated by finite unions of such rectangles, the following corollary is an immediate consequence of Theorem \ref{thm:dirichlet}.

\begin{corollary} \label{cor1}
Let $k \geq 2$ be a fixed integer. For $x \geq 1,$ let $n$ be a random integer chosen uniformly from $[1,x]$ and $(d_1,\ldots,d_k)$ be a random $k$-tuple chosen uniformly from the set of all possible factorization $\{(m_1,\ldots,m_k) \in \mathbb{N}^k\, : \, n=m_1\cdots m_k\}.$ Then we have the convergence in distribution
\[\left(\frac{\log d_1}{\log n}, \ldots, \frac{\log d_k}{\log n}\right) \xrightarrow[]{d}
\mathrm{Dir}\left(\frac{1}{k},\ldots,\frac{1}{k} \right)\]
as $x \to \infty.$
\end{corollary}
It is a general phenomenon that the ``anatomy'' of polynomials or permutations is essentially the same as that of integers (see \cite{granville2008anatomy}, \cite{MR3966460}), and the main theorem here is no exception. 
In the realm of polynomials, the following theorem serves as the counterpart to 
Theorem \ref{thm:dirichlet}.

\begin{theorem} \label{thm:dirichlet2} Let $k\geq 2$ be a fixed integer and $q$ be a fixed prime power. Then uniformly for $n \geq 1$ and $u_1,\ldots,u_{k-1}\geq 0$ satisfying $u_1+\cdots+u_{k-1}\leq 1$, we have
\begin{align}
\frac{1}{q^n}\sum_{F \in \mathcal{M}_q(n)}\frac{1}{\tau_k(F)}\underset{D_1\cdots D_{k-1} \mid F}{\sum_{\substack{D_1 \in \mathcal{M}_q\\\deg D_1  \leq nu_1}}\cdots
\sum_{\substack{D_{k-1} \in \mathcal{M}_q\\\deg D_{k-1}  \leq nu_{k-1}}}}1 
=F(u_1,\ldots,u_{k-1})
+O\left(n^{-\frac{1}{k}} \right),
\end{align}
where the notations are defined in 
Section \ref{section2}.
\end{theorem}
\begin{corollary} 
Let $k \geq 2$ be a fixed integer and $q$ be a fixed prime power. For $n \geq 1,$ let $F$ be a random polynomial chosen uniformly from $\mathcal{M}_q(n)$ and $(D_1,\ldots,D_k)$ be a random $k$-tuple chosen uniformly from the set of all possible factorization $\{(G_1,\ldots,G_k) \in \mathcal{M}_q^k\, : \, F=G_1\cdots G_k\}.$ Then we have the convergence in distribution
\[\left(\frac{\deg D_1}{n}, \ldots, \frac{\deg D_{k}}{n}\right) \xrightarrow[]{d}
\mathrm{Dir}\left(\frac{1}{k},\ldots,\frac{1}{k} \right)\]
as $n \to \infty.$
\end{corollary}

Similarly, in the realm of permutations, the following theorem serves as the counterpart to 
Theorem \ref{thm:dirichlet}.

\begin{theorem} \label{thm:dirichlet3} Let $k\geq 2$ be a fixed integer. Then uniformly for $n \geq 1$ and $u_1,\ldots,u_{k-1}\geq 0$ satisfying $u_1+\cdots+u_{k-1}\leq 1$, we have
\begin{align}
\frac{1}{n!}\sum_{\sigma \in S_n}\frac{1}{\tau_k(\sigma)}
\underset{{\substack{[n]=A_1 \sqcup \cdots \sqcup A_k\\ \sigma(A_i)=A_i, i=1,\ldots,k \\ 0 \leq |A_i|  \leq nu_i,  i=1,\ldots,k-1}}}{\sum\cdots\sum}1
=F(u_1,\ldots,u_{k-1})+O\left(n^{-\frac{1}{k}}\right),
\end{align}
where the notations are defined in 
Section \ref{section2}.
\end{theorem}

\begin{corollary} 
Let $k \geq 2$ be a fixed integer. For $n \geq 1,$ let $\sigma$ be a random permutation chosen uniformly from $S_n$ and $(A_1,\ldots, A_k)$ be a random $k$-tuple chosen uniformly from the set of all possible $\sigma$-invariant decomposition $\{(B_1,\ldots, B_k) \, : \, [n]=B_1\sqcup \cdots \sqcup B_k, \sigma(B_i)=B_i  \text{ for }i=1,\ldots, k\}.$ Then we have the convergence in distribution
\[\left(\frac{|A_1|}{n},\ldots, \frac{|A_k|}{n}\right) \xrightarrow[]{d}
\mathrm{Dir}\left(\frac{1}{k},\ldots,\frac{1}{k} \right)\]
as $n \to \infty.$
\end{corollary}

In Section \ref{section6}, we will model the Dirichlet distribution with arbitrary parameters by assigning probability weights which are not necessarily uniform to each integer and to each factorization. Then, as we will see, most of the results in the literature about the distribution of divisors in logarithmic scale are direct consequences of Theorem \ref{thmgeneral}, which is a generalization of Theorem \ref{thm:dirichlet}.

\section{Notation} \label{section2}

Throughout the paper, we shall adopt the following list of notations:
\begin{itemize}
\item we say $f(x)=O(g(x))$ or $f(x) \ll g(x)$ if there exists a constant $C>0$ which might depend on $k,q, \boldsymbol{\alpha}, \boldsymbol{\beta}, \boldsymbol{c}, \boldsymbol{\delta}$ such that $|f(x)| \leq C \cdot g(x)$ whenever $x >x_0$ for some $x_0>0;$


\item $[n]:=\{1,2,\ldots, n\};$

\item $\tau_k(n):=|\{(d_1,\ldots, d_k) \in \mathbb{N}^k \, : \, n=d_1\cdots d_{k} \} |$ and $\tau(n):=\tau_2(n);$

\item $\mathcal{M}_q:=\{F \in \mathbb{F}_q[x] \, : \, F \text{ is monic}\};$

\item $\mathcal{M}_q(n):=\{F \in \mathcal{M}_q \, : \, \deg{F}=n\};$


\item $\tau_k(F):=|\{(D_1,\ldots, D_k) \in \mathcal{M}_q^k \, : \, F=D_1\cdots D_{k} \} |;$

\item $S_n$ denotes the group of permutations on $[n];$

\item $c(\sigma)$ denotes the number of disjoint cycles of the permutation $\sigma;$

\item $\tau_{\alpha}(\sigma):={\alpha}^{c(\sigma)};$

\item $\left[{n \atop k}\right]:=|\{\sigma \in S_n \, : \, c(\sigma)=k\}|$ denote (unsigned) Stirling numbers of the first kind.

\end{itemize}

\section{Properties of $\mathcal{D}(s_1,\ldots, s_k)$}

Both \cite{tenenbaum1980lois} and \cite{nyandwi2014distribution} deal with the divisors one by one using  \cite[Th\'eor\`eme~ T]{tenenbaum1980lois} followed by partial summation. However, as $k$ gets larger, it is increasingly laborious to achieve full uniformity as well as the optimal rate of convergence, especially when one of the $u$'s is small or $u_1+\cdots+u_{k-1}$ is close to $1$. Instead, we would like to apply Mellin's inversion formula to (the second derivative of) the multiple Dirichlet series $\mathcal{D}(s_1,\ldots, s_k)$ defined below, which allows a more symmetric approach to the problem so that all the divisors can be handled simultaneously. We first establish a few properties of the multiple Dirichlet series that are essential to the proof of Theorem \ref{thm:dirichlet}.

\begin{lemma}\label{lem:interchange} Let $\mathcal{D}(s_1,\ldots,  s_k)$ denote the multiple Dirichlet series 
\begin{align*} 
\sum_{n_1=1}^{\infty}\cdots
\sum_{n_k=1}^{\infty}\frac{\tau_k(n_1\cdots n_k)^{-1}}{n_1^{s_1}\cdots n_k^{s_k}}.
\end{align*}
Then $\mathcal{D}(s_1,\ldots, s_k)$ converges absolutely in the domain
\begin{align*}
\Omega:=\{(s_1,\dots,s_k)\in \mathbb{C}^k \, : \, \text{$\Re(s_j)>1$ {\normalfont for} $j=1,\dots,k$}\}
\end{align*}
and uniformly on any compact subset of $\Omega.$ In particular, $\mathcal{D}(s_1,\dots,s_k)$ is an analytic function of $k$ variables in $\Omega.$
\end{lemma}

\begin{proof}
Let $\sigma_j:=\Re(s_j)$ for $j=1,\ldots,k.$ Then, since
\begin{align*}
\sum_{n_1=1}^{\infty}\cdots
\sum_{n_k=1}^{\infty}\left|\frac{\tau_k(n_1\cdots n_k)^{-1}}{n_1^{s_1}\cdots n_k^{s_k}}\right| \leq \zeta(\sigma_1)\cdots\zeta(\sigma_k) <\infty,
\end{align*}
the lemma follows.
\end{proof}


\begin{lemma}\label{lem:euler}
The multiple Dirichlet series $\mathcal{D}(s_1,\dots,s_k)$ can be expressed as the Euler product
\begin{align*}
\prod_p\sum_{v_1=0}^{\infty}\cdots\sum_{v_k=0}^{\infty}{v_1+\cdots+v_k+k-1 \choose {k-1}}^{-1}
p^{-(v_1s_1+\cdots+v_ks_k)}
\end{align*}
in the domain $\Omega$ defined above.
\end{lemma}

\begin{proof}

Let $y \geq 2$ and $\sigma_j:=\Re(s_j)$ for $j=1,\dots,k.$ Then, since
\begin{align*}
\sum_{v_1=0}^{\infty}\cdots\sum_{v_k=0}^{\infty}{v_1+\cdots+v_k+k-1 \choose {k-1}}^{-1}
p^{-(v_1\sigma_1+\cdots+v_k\sigma_k)}
\leq \left(1-\frac{1}{p^{\sigma_1}} \right)^{-1}
\cdots \left(1-\frac{1}{p^{\sigma_k}} \right)^{-1}<\infty,
\end{align*}
the finite product
\begin{align*}
\prod_{p \leq y}\sum_{v_1=0}^{\infty}\cdots\sum_{v_k=0}^{\infty}{v_1+\cdots+v_k+k-1 \choose {k-1}}^{-1}
p^{-(v_1s_1+\cdots+v_ks_k)}
\end{align*}
is well-defined.\\
Let $S(y):=\{n\geq 1 \,: \, p|n \text{ implies } p \leq y\}$ be the set of $y$-smooth numbers. Then, since 
\begin{align*}
\tau_k(p^v)={v+k-1 \choose k-1},
\end{align*}
we have
\begin{align*}
&\left|\prod_{p \leq y}\sum_{v_1=0}^{\infty}\cdots\sum_{v_k=0}^{\infty}{v_1+\cdots+v_k+k-1 \choose {k-1}}^{-1}
p^{-(v_1s_1+\cdots+v_ks_k)}- \mathcal{D}(s_1,\dots,s_k) \right|\\
=&\left|\sum_{n_1 \in S(y)}\cdots\sum_{n_k \in S(y)}\frac{\tau_k(n_1\cdots n_k)^{-1}}{n_1^{s_1}\cdots n_k^{s_k}} - \mathcal{D}(s_1,\dots,s_k)\right| 
\leq \sum_{j=1}^k \prod_{\substack{i=1\\i\neq j}}^k \zeta(\sigma_i)
\sum_{n_j>y}\frac{1}{n_j^{\sigma_j}}.
\end{align*}
The lemma follows by letting $y \to \infty.$
\end{proof}

\begin{lemma}\label{lem:growth}
For $j=1,\ldots,k,$ let $R_j \subseteq \left\{ s_j \in \mathbb{C} \,:\, \Re(s_j)  > \frac{3}{4}, \, |\Im(s_j)| > \frac{1}{4} \right\}$ be a 
zero-free region for $\zeta(s_j).$
Then the multiple Dirichlet series $\mathcal{D}(s_1,\dots,s_k)$ can be continued analytically to the domain $\prod_{j=1}^k R_j.$ Moreover, we have the bound
\begin{align} \label{eq:middlebound}
\mathcal{D}(s_1,\dots,s_k) \ll  |\zeta(s_1)|^{\frac{1}{k}}\cdots|\zeta(s_k)|^{\frac{1}{k}}.
\end{align}
\end{lemma}

\begin{proof}
Let $(s_1,\dots,s_k)\in \mathbb{C}^k$ with $\sigma_j:=\Re(s_j)>1$ for $j=1,\dots,k.$ Then by Lemma \ref{lem:euler} we have the Euler product expression
\begin{align*}
&\zeta(s_1)^{-\frac{1}{k}}\cdots\zeta(s_k)^{-\frac{1}{k}}\mathcal{D}(s_1,\dots,s_k)\\
=&\prod_p\left(\prod_{j=1}^{k}\left(1-\frac{1}{p^{s_j}}\right)^{\frac{1}{k}}\right)\sum_{v_1=0}^{\infty}\cdots\sum_{v_k=0}^{\infty}{v_1+\cdots+v_k+k-1 \choose {k-1}}^{-1}
p^{-(v_1s_1+\cdots+v_ks_k)},
\end{align*}
where the $k$-th root is understood as its principal branch.\\
For $j=1,\ldots,k,$ expanding the $k$-th root as
\begin{align*}
\sum_{r=0}^{\infty}(-1)^{r}{\frac{1}{k} \choose r}p^{-rs_j},
\end{align*}
we find that the factors of the Euler product are 
$1+O\left(\sum_{i=1}^{k}\sum_{j=1}^{k} p^{-(\sigma_i+\sigma_j)} \right) $ by Taylor's theorem. Therefore, the function
\begin{align}\label{analytic}
\zeta(s_1)^{-\frac{1}{k}}\cdots\zeta(s_k)^{-\frac{1}{k}}\mathcal{D}(s_1,\dots,s_k)
\end{align}
can be continued analytically to the domain where
$\Re(s_j) > \frac{3}{4}$ for $j=1,\dots,k,$ in which it is
uniformly bounded.

On the other hand, for $(s_1,\dots,s_k) \in \prod_{j=1}^k R_j,$ we can express $\mathcal{D}(s_1,\dots,s_k)$ as
\begin{align*}
\zeta(s_1)^{\frac{1}{k}}\cdots\zeta(s_k)^{\frac{1}{k}}\left(\zeta(s_1)^{-\frac{1}{k}}\cdots\zeta(s_k)^{-\frac{1}{k}}\mathcal{D}(s_1,\dots,s_k)\right),
\end{align*}
and so the lemma follows.
\end{proof}

\begin{lemma}\label{power} 
In the open hypercube 
\begin{align*}
Q:=\left\{(s_1,\ldots,s_k) \in \mathbb{C}^k \, : \,  1<\Re(s_j)<7/4,\, |\Im(s_j)|<3/4 \text{ {\normalfont for} $j=1,\dots,k$}\right\},
\end{align*}
we have the estimate
\begin{align*}
\frac{\partial^{2k}}{\partial s_1^2\cdots\partial s_k^2}\mathcal{D}(s_1,\dots,s_k)
=\left(1+\frac{1}{k}\right)^k\frac{1}{k^k}&(s_1-1)^{-\frac{1}{k}-2}\cdots(s_k-1)^{-\frac{1}{k}-2}\\
\cdot &(1+O(|s_1-1|+\cdots+|s_k-1|)).
\end{align*}
\end{lemma}

\begin{proof}
By (\ref{analytic}) and the fact that
\begin{align} \label{zeta at s=1}
\zeta(s)=\frac{1}{s-1}+O(1),
\end{align}
we have the power series representation
\begin{align*}
(s_1-1)^{\frac{1}{k}}\cdots(s_k-1)^{\frac{1}{k}}
\mathcal{D}(s_1,\dots,s_k)=
\sum_{(i_1,\dots,i_k)\in \mathbb{Z}_{\geq 0}^k}
a_{i_1,\dots,i_k}(s_1-1)^{i_1}\cdots(s_k-1)^{i_k}
\end{align*}
in $Q$ for some constants $a_{i_1,\dots,i_k} \in \mathbb{C}.$
It follows that
\begin{align*}
\frac{\partial^{2k}}{\partial s_1^2\cdots\partial s_k^2}\mathcal{D}(s_1,\dots,s_k)
=\left(1+\frac{1}{k}\right)^k\frac{1}{k^k}&(s_1-1)^{-\frac{1}{k}-2}\cdots(s_k-1)^{-\frac{1}{k}-2}\nonumber\\
\cdot & (a_{\boldsymbol{0}}+O(|s_1-1|+\cdots+|s_k-1|)),
\end{align*}
where by (\ref{zeta at s=1}) and Lemma \ref{lem:euler},
the leading coefficient
\begin{align*}
a_{\boldsymbol{0}}&=\prod_p\left(1-\frac{1}{p}\right)\sum_{v_1=0}^{\infty}\cdots\sum_{v_k=0}^{\infty}{v_1+\cdots+v_k+k-1 \choose {k-1}}^{-1}
p^{-(v_1+\cdots+v_k)}\\
&=\prod_p \left(1-\frac{1}{p} \right)\sum_{v=0}^{\infty}{v+k-1 \choose k-1}
{v+k-1 \choose k-1}^{-1}p^{-v}=1.
\end{align*}
\end{proof}

\section{Proof of Theorem \ref{thm:dirichlet}}

We begin with writing
\begin{align*}
\sum_{n \leq x}\frac{1}{\tau_k(n)}\underset{d_1\cdots d_{k-1} \mid n}{\sum_{d_1 \leq n^{u_1}}\cdots
\sum_{d_{k-1} \leq n^{u_{k-1}}}}1 
=\sum_{n \leq x}\frac{1}{\tau_k(n)}
\left(\underset{d_1\cdots d_{k-1} \mid n}{\sum_{d_1 \leq x^{u_1}}\cdots
\sum_{d_{k-1} \leq x^{u_{k-1}}}}1
-\underset{\substack{d_1\cdots d_{k-1} \mid n
\\ n^{u_j}<d_j\leq x^{u_j} \text{ for some $j$}}}{\sum_{d_1 \leq x^{u_1}}\cdots
\sum_{d_{k-1} \leq x^{u_{k-1}}}}1
\right),
\end{align*}
where the main term is
\begin{align}\label{main}
\sum_{n \leq x}\frac{1}{\tau_k(n)}\underset{d_1\cdots d_{k-1} \mid n}{\sum_{d_1 \leq x^{u_1}}\cdots
\sum_{d_{k-1} \leq x^{u_{k-1}}}}1
=\sum_{d_1 \leq x^{u_1}}\cdots
\sum_{d_{k-1} \leq x^{u_{k-1}}}\sum_{d_k \leq x/(d_1\cdots d_{k-1})}\frac{1}{\tau_k(d_1\cdots d_{k})},
\end{align}
and the error term is
\begin{align}\label{error}
\sum_{n \leq x}\frac{1}{\tau_k(n)}\underset{\substack{d_1\cdots d_{k-1} \mid n
\\ n^{u_j}<d_j\leq x^{u_j}\text{ for some $j$}}}{\sum_{d_1 \leq x^{u_1}}\cdots
\sum_{d_{k-1} \leq x^{u_{k-1}}}}1\leq \sum_{\substack{j=1\\u_j \neq 0}}^{k-1}\sum_{n \leq x}\frac{1}{\tau_k(n)}
\sum_{n^{u_j}<d_j\leq x^{u_j}}\underset{\substack{d_i \leq x^{u_i}, i \neq j\\d_1\cdots d_{k-1} \mid n}}{\sum\cdots\sum}\quad
1.
\end{align}

Let us first bound the error term (\ref{error}).
For $j=1,\dots,k-1$ with $u_j \neq 0,$ 
we write $n=d_jm$ for some integer $m.$ Then $d_j>n^{u_j}$ implies $m<d_j^{(1-u_j)/u_j},$ and the number of ways of obtaining $m$ as a product $d_1\cdots d_{j-1}d_{j+1}\cdots d_k$ is bounded by $\tau_{k-1}(m).$ It follows that
\begin{align} \label{finaleerror}
\sum_{n \leq x}\frac{1}{\tau_k(n)}
\sum_{n^{u_j}<d_j\leq x^{u_j}}\underset{\substack{d_i \leq x^{u_i}
\text{ for }i \neq j\\d_1\cdots d_{k-1} \mid n}}{\sum\cdots\sum}\quad
1 \leq
\sum_{d_j\leq x^{u_j}}\sum_{m<d_j^{(1-u_j)/u_j}}\frac{\tau_{k-1}(m)}{\tau_k(d_jm)}.
\end{align}
If $u_{j}\leq \frac{1}{2},$ then using \cite[Theorem~14.2]{koukoulopoulos2020distribution}, this is bounded by
\begin{align}
\sum_{d_{j}\leq x^{u_{j}}} \sum_{m < d_{j}^{(1-u_{j})/u_{j}}}
\frac{\tau_{k-1}(m)}{\tau_k(m)} 
&\ll \sum_{2\leq d_{j}\leq x^{u_{j}}} \frac{d_{j}^{(1-u_{j})/u_{j}}}{\left(\log d_{j}^{(1-u_{j})/u_{j}}\right)^{\frac{1}{k}}} \nonumber \\
&\ll \frac{x}{(\log x)^{\frac{1}{k}}}. \nonumber
\end{align}
Otherwise, it follows from the simple observation 
\begin{align}\label{observation}
\frac{\tau_{k-1}(m)}{\tau_k(d_jm)} &\leq
\frac{\tau_{k-1}(d_{j}m)}{\tau_k(d_{j}m)} 
\leq \frac{\tau_{k-1}(d_{j})}{\tau_k(d_{j})} 
\end{align}
that the expression (\ref{finaleerror}) is bounded by
\begin{align}
\sum_{d_{j}\leq x^{u_{j}}}
d_{j}^{(1-u_{j})/u_{j}}
\frac{\tau_{k-1}(d_{j})}{\tau_k(d_{j})}
&\ll x^{1-u_{j}} \sum_{d_{j}\leq x^{u_{j}}} \frac{\tau_{k-1}(d_{j})}{\tau_k(d_{j})} \nonumber\\
&\ll \frac{x}{(\log x)^{\frac{1}{k}}} \nonumber
\end{align}
again using \cite[Theorem~14.2]{koukoulopoulos2020distribution}.

Now we are left with the main term (\ref{main}). In order to apply Mellin's inversion formula, we follow the treatment in \cite{granville2019beyond} and \cite[Chapter~13]{koukoulopoulos2020distribution}.

\begin{lemma} \label{mellin}
Let $T \geq 1.$ Let $\phi, \psi \colon [0, \infty) \to \mathbb{R}$ be smooth functions supported on $[0, 1]$ and $\left[0, 1+\frac{1}{T}\right]$ respectively with
\begin{align*} 
\begin{cases} 
\phi(y)=1 & \mbox{if } y \leq 1-\frac{1}{T},\\
\phi(y) \in [0,1] & \mbox{if } 1-\frac{1}{T}<y \leq 1,\\
\phi(y) = 0 & \mbox{if } y > 1,
\end{cases}
\end{align*}
and
\begin{align*}
\begin{cases} 
\psi(y)=1 & \mbox{if } y \leq 1,\\
\psi(y) \in [0,1] & \mbox{if } 1<y \leq 1+\frac{1}{T},\\
\psi(y) = 0 & \mbox{if } y > 1+\frac{1}{T}.
\end{cases}
\end{align*}
Moreover, for each integer $j\geq 0$, their derivatives  satisfy  the growth condition $\phi^{(j)}(y), \psi^{(j)}(y) \ll_j T^j$ uniformly for $y\geq 0.$ 
Let $\Phi(s), \Psi(s)$ be the Mellin transform of $\phi(y), \psi(y)$ respectively for $1 \leq \Re(s) \leq 2$, i.e.
\begin{align*}
\Phi(s)=\int_{0}^{\infty}\phi(y)y^s\frac{dy}{y},
\end{align*}
and
\begin{align*}
\Psi(s)=\int_{0}^{\infty}\psi(y)y^s\frac{dy}{y}.
\end{align*}
Then we have the estimates
\begin{align} \label{smallt}
\Phi(s), \Psi(s)=\frac{1}{s}+O\left(\frac{1}{T}\right),
\end{align}
and
\begin{align} \label{larget}
\Phi(s), \Psi(s) \ll_j \frac{T^{j-1}}{|s|^j}
\end{align}
for $j\geq 1.$
\end{lemma}
\begin{proof}
See \cite[Theorem~4]{granville2019beyond}.
\end{proof}

We need the following version of Hankel's lemma to extract the main contribution from the multidimensional contour integral in the proof of Lemma \ref{mainlemma}.

\begin{lemma} \label{Hankel}
Let $x>1, \sigma>1$ and $\Re(\alpha)>1.$ Then we have
\begin{align*}
\frac{1}{2\pi i}\int_{\Re(s)=\sigma}\frac{x^s}{s(s-1)^{\alpha}}ds=\frac{1}{\Gamma(\alpha)}\int_{1}^x(\log y)^{\alpha-1}dy.
\end{align*}
\end{lemma}
\begin{proof}
See \cite[Lemma~13.1]{koukoulopoulos2020distribution}.
\end{proof}

We now prove the main lemma. 
\begin{lemma} \label{mainlemma}
Let $x_1,\dots,x_k \geq e$ and $S(x_1,\dots,x_k)$ denote the weighted sum
\begin{align*}
\sum_{d_1 \leq x_1}\cdots\sum_{d_k \leq x_k}\frac{(\log d_1)^2 \cdots (\log d_k)^2}{\tau_k(d_1\cdots d_k)}.
\end{align*}
Then we have
\begin{align*}
S(x_1,\dots,x_k)=\frac{1}{\Gamma\left(\frac{1}{k} \right)^k}\prod_{j=1}^{k}\int_1^{x_j}(\log y_j)^{\frac{1}{k}+1}dy_j+R(x_1,\dots,x_k)
\end{align*}
with
\begin{align*}
R(x_1,\dots,x_k) \ll x_1\cdots x_k\sum_{j=1}^k \left(\prod_{\substack{i=1\\i \neq j}}^k  (\log x_i)^{\frac{1}{k}+1}\right) (\log x_j)^{\frac{1}{k}}.
\end{align*}
\end{lemma}
As in \cite{granville2019beyond} and \cite[Chapter~13]{koukoulopoulos2020distribution}, we introduce powers of logarithms to ensure that the major contribution to the multiple Perron integral below comes from $s_1,\ldots,s_k \approx 1.$ Later on, they will be removed by partial summation.
\begin{proof}
The proof consists of four steps: Mellin inversion, localization, approximation and completion. For $j=1,\ldots,k$, let $T_j=2(\log x_j)^{2}$ and $\phi_j, \psi_j$ be any smooth functions coincide with $\phi, \psi$ respectively from Lemma \ref{mellin}.
Then the weighted sum $S(x_1,\ldots,x_k)$ is bounded between
\begin{align*} 
\sum_{d_1=1}^{\infty}\cdots\sum_{d_k =1}^{\infty}\frac{(\log d_1)^2 \cdots (\log d_k)^2}{\tau_k(d_1\cdots d_k)}\phi_1\left(\frac{d_1}{x_1} \right)\cdots\phi_k\left(\frac{d_k}{x_k}\right),
\end{align*}
and
\begin{align} \label{upperupper}
\sum_{d_1=1}^{\infty}\cdots\sum_{d_k =1}^{\infty}\frac{(\log d_1)^2 \cdots (\log d_k)^2}{\tau_k(d_1\cdots d_k)}\psi_1\left(\frac{d_1}{x_1} \right)\cdots\psi_k\left(\frac{d_k}{x_k}\right).
\end{align}
To avoid repetitions, we only establish the upper bound here.
Applying Mellin's inversion formula, the expression (\ref{upperupper}) becomes
\begin{align*} 
\sum_{d_1 =1}^{\infty}\cdots\sum_{d_k =1}^{\infty}\frac{(\log d_1)^2 \cdots (\log d_k)^2}{\tau_k(d_1\cdots d_k)}
\prod_{j=1}^{k}\left(\frac{1}{2\pi i}\int_{\Re(s_j)=1+\frac{1}{2\log x_j}}\Psi_j(s_j)\left( \frac{d_j}{x_j}\right)^{-s_j}ds_j \right).
\end{align*}
Then, by Lemma \ref{lem:interchange} and Lemma {\ref{mellin}}, it is valid to interchange the order of summation and integration, and so this becomes
\begin{align} \label{afterinterchange}
\frac{1}{(2\pi i)^k}\int_{\Re(s_1)=1+\frac{1}{2\log x_1}}\cdots\int_{\Re(s_k)=1+\frac{1}{2\log x_k}}
&\left(\frac{\partial^{2k}}{\partial s_1^2\cdots\partial s_k^2}\mathcal{D}(s_1,\dots,s_k)\right) \nonumber\\
\cdot&\Psi_1(s_1)\cdots\Psi_k(s_k)x_1^{s_1}\cdots x_k^{s_k}ds_1\cdots ds_k.
\end{align}

For each $j=1,\dots,k,$ we decompose the vertical contour $ I_j:=\left\{s_j \in \mathbb{C}\, : \,\Re(s_j)=1+\frac{1}{2\log x_j}\right\}$ as $I_{j}^{(1)} \cup I_{j}^{(2)} \cup I_{j}^{(3)}$ (traversed upwards),  where
\begin{align*}
I_{j}^{(1)}:=\left\{s_j \in I_j \, : \, |\Im(s_j)|\leq 1/2\right\},
\end{align*}
\begin{align*}
I_{j}^{(2)}:=\left\{s_j \in I_j \, : \, 1/2<|\Im(s_j)|\leq  T_j^2/2 \right\},
\end{align*}
and
\begin{align*}
I_{j}^{(3)}:=\left\{s_j \in I_j \, : \, |\Im(s_j)|>  T_j^2/2\right\}.
\end{align*}
To establish an upper bound on the second derivative of the multiple Dirichlet series, we shall apply Cauchy's integral formula for derivatives of $k$ variables. For this purpose, we invoke Lemma \ref{lem:growth} with the classical zero-free region 
\begin{align*}
R_j:=\left\{s_j \in \mathbb{ C} \,:\, \Re(s_j)> 
\begin{cases}
1-\frac{c}{\log T_j} & \mbox{{\normalfont if} } \frac{1}{4}<|\Im(s_j)|< {T_j}^{2},\\
1 & \mbox{{\normalfont otherwise }} 
\end{cases}
\right\}
\end{align*}
with $c=1/100$ say, for $j=1,\ldots,k.$ Moreover, we introduce the distinguished boundary
\begin{align*}
\Gamma_{s_1,\dots,s_k}:=\left\{(w_1,\dots,w_k) \in \mathbb{C}^k \, : \, 
|w_j-s_j|=
\begin{cases}
\frac{|s_j-1|}{2} & \mbox{if } s_j \in I_j^{(1)},\\
\frac{c}{4\log T_j} & \mbox{if } s_j \in I_j^{(2)},\\
\frac{1}{4\log x_j} & \mbox{if } s_j \in I_j^{(3)}
\end{cases}
\text{ for $j=1,\dots,k$}
\right\}
\end{align*}
as there will be different bounds on $\zeta(w_j)$ depending on the height.
Then, Cauchy's formula implies 
\begin{align*}
\frac{\partial^{2k}}{\partial s_1^2\cdots\partial s_k^2}\mathcal{D}(s_1,\dots,s_k) 
&\ll \frac{\sup_{(w_1,\dots,w_k) \in \Gamma_{s_1,\dots,s_k}}|\mathcal{D}(w_1,\dots,w_k)|}
{\left(\prod_{\substack{j:s_j \in I_{j}^{(1)}}}|s_j-1|^2\right)
\left(\prod_{j:s_j \in I_{j}^{(2)}}(\log T_j)^{-2}\right)\left(
\prod_{j:s_j \in I_{j}^{(3)}}
(\log x_j)^{-2}\right)} 
\end{align*}
with 
\begin{align*}
\mathcal{D}(w_1,\dots,w_k)\ll |\zeta(w_1)|^{\frac{1}{k}}\cdots|\zeta(w_k)|^{\frac{1}{k}}
\end{align*}
given by (\ref{eq:middlebound}) from Lemma \ref{lem:growth}. Using (\ref{zeta at s=1}), \cite[Theorem~3.5]{titchmarsh1986theory} that
\begin{align*}
\zeta(w_j) \ll \log T_j
\end{align*}
whenever $\frac{1}{4}\leq|\Im(w_j)| \leq T_j^2,$ and the simple upper bound
\begin{align*}
|\zeta(w_j)| \leq \sum_{n=1}^{\infty}\frac{1}{n^{1+\frac{1}{2\log x_j}}} \ll \log x_j,
\end{align*}
we arrive at the derivative bound
\begin{align} \label{derivativebound}
\frac{\partial^{2k}}{\partial s_1^2\cdots\partial s_k^2}\mathcal{D}(s_1,\dots,s_k)  &\ll
\prod_{j:s_j \in I_{j}^{(1)}}|s_j-1|^{-\frac{1}{k}-2}
\prod_{j:s_j \in I_{j}^{(2)}}(\log T_j)^{\frac{1}{k}+2}
\prod_{j:s_j \in I_{j}^{(3)}}
(\log x_j)^{\frac{1}{k}+2}.
\end{align}
 Applying (\ref{larget}) with $j=1,2$ from Lemma \ref{mellin}, for $j=1,\dots,k,$ we have the estimates 
\begin{align*} 
(\log x_j)^{\frac{1}{k}+2}\int_{I_{j}^{(3)}}|\Psi(s_j)x_j^{s_j}ds_j| &\ll x_j(\log x_j)^{\frac{1}{k}+2}T_j\int_{T_j^2/2}^{\infty}\frac{dt}{t^2}\\
&\ll \frac{x_j(\log x_j)^{\frac{1}{k}+2}}{T_j},
\end{align*}
\begin{align*} 
(\log T_j)^{\frac{1}{k}+2}\int_{I_{j}^{(2)}}|\Psi(s_j){x_j}^{s_j}ds_j| &\ll x_j(\log T_j)^{\frac{1}{k}+2}\int_{1/2}^{T_j^2/2}\frac{dt}{t}\\
&\ll x_j(\log T_j)^{\frac{1}{k}+3},
\end{align*}
and
\begin{align} \label{I1bound}
\int_{I_{j}^{(1)}}|(s_j-1)^{-\frac{1}{k}-2}\Psi(s_j){x_j}^{s_j}ds_j|
&\ll \int_{I_{j}^{(1)}}\left|(s_j-1)^{-\frac{1}{k}-2}{x_j}^{s_j}\frac{ds_j}{s_j}\right|
 \nonumber\\
&\ll x_j\int_{-1/2}^{1/2}\left|\frac{1}{2\log x_j}+it \right|^{-\frac{1}{k}-2}dt \nonumber\\
&\ll x_j(\log x_j)^{\frac{1}{k}+1}.
\end{align}
Therefore, combining with (\ref{smallt}) from
Lemma \ref{mellin} and 
(\ref{derivativebound}), the main contribution to (\ref{afterinterchange}) is 
\begin{align} \label{I1}
\frac{1}{(2\pi i)^k}\int_{I_{1}^{(1)}}\cdots\int_{I_{k}^{(1)}}
\left(\frac{\partial^{2k}}{\partial s_1^2\cdots\partial s_k^2}\mathcal{D}(s_1,\dots,s_k)\right)x_1^{s_1}\cdots x_k^{s_k}
\frac{ds_1}{s_1}\cdots\frac{ds_k}{s_k}
\end{align}
with an error term 
\begin{align} \label{strongerror}
\ll x_1\cdots x_k\sum_{j=1}^k \left(\prod_{\substack{i=1\\i \neq j}}^k  (\log x_i)^{\frac{1}{k}+1}\right) (\log x_j)^{\frac{1}{k}}
\end{align}
as $T_j= 2(\log x_j)^{2}$ for $j=1,\ldots,k.$

 Applying Lemma \ref{power}, the main contribution to (\ref{I1}) is
\begin{align} \label{mainnew}
\left(1+\frac{1}{k}\right)^k\frac{1}{k^k}\prod_{j=1}^k\left(\frac{1}{2\pi i}\int_{I_j^{(1)}}(s_j-1)^{-\frac{1}{k}-2}x_j^{s_j}\frac{ds_j}{s_j}\right)
\end{align}
with an error term 
\begin{align} \label{errornew}
\ll \sum_{j=1}^k \int_{I_{j}^{(1)}}\left|(s_j-1)^{-\frac{1}{k}-1}x_j^{s_j}\frac{ds_j}{s_j} \right| \prod_{\substack{i=1\\i \neq j}}^k 
\int_{I_{i}^{(1)}} \left|(s_i-1)^{-\frac{1}{k}-2}x_i^{s_i}\frac{ds_i}{s_i} \right|.
\end{align}
For $j=1,\dots,k,$ we have 
\begin{align*}
\int_{I_{j}^{(1)}} \left|(s_j-1)^{-\frac{1}{k}-1}x_j^{s_j}\frac{ds_j}{s_j} \right|
&\ll x_j\int_{-1/2}^{1/2} \left|\frac{1}{2\log x_j}+it \right|^{-\frac{1}{k}-1}dt\\
&\ll x_j(\log x_j)^{\frac{1}{k}}.
\end{align*}
Combining with (\ref{I1bound}), the expression (\ref{errornew}) is 
\begin{align} \label{error1}
\ll x_1\cdots x_k\sum_{j=1}^k  \left(\prod_{\substack{i=1\\i \neq j}}^k (\log x_i)^{\frac{1}{k}+1}\right)(\log x_j)^{\frac{1}{k}}.
\end{align}

Since for $j=1,\dots,k$ we have the bound
\begin{align*}
\int\limits_{\substack{\Re(s_j)=1+\frac{1}{2\log x_j}\\|\Im(s_j)|>\frac{1}{2}}} \left|(s_j-1)^{-\frac{1}{k}-2}x_j^{s_j}\frac{ds_j}{s_j} \right|
&\ll x_j\int_{1/2}^{\infty} t^{-\frac{1}{k}-3}dt\\
&\ll x_j,
\end{align*}
it follows from (\ref{I1bound}) that the main contribution to (\ref{mainnew}) is 
\begin{align} \label{mainnewnew}
\left(1+\frac{1}{k}\right)^k\frac{1}{k^k} \prod_{j=1}^k \left(\frac{1}{2\pi i} \int_{\Re(s_j)=1+\frac{1}{2\log x_j}}(s_j-1)^{-\frac{1}{k}-2}x_j^{s_j}\frac{ds_j}{s_j}\right)
\end{align}
with an error term 
\begin{align} \label{errornewnew}
\ll x_1\cdots x_k\sum_{j=1}^k \prod_{\substack{i=1\\i \neq j}}^k (\log x_i)^{\frac{1}{k}+1}.
\end{align}
 Applying Lemma \ref{Hankel}, for $j=1,\dots,k$ we have
\begin{align*} 
\frac{1}{2\pi i}\int_{\Re(s_j)=1+\frac{1}{2\log x_j}}(s_j-1)^{-\frac{1}{k}-2}x_j^{s_j}\frac{ds_j}{s_j}&=\frac{1}{\Gamma \left(\frac{1}{k}+2 \right)}\int_{1}^{x_j}(\log y_j)^{\frac{1}{k}+1}dy_j\nonumber\\
&=\left(1+\frac{1}{k}\right)^{-1} \frac{k}{\Gamma \left(\frac{1}{k} \right)}\int_{1}^{x_j}(\log y_j)^{\frac{1}{k}+1}dy_j.
\end{align*}
Finally, the lemma follows from collecting the main term (\ref{mainnewnew}) and the error terms (\ref{strongerror}), (\ref{error1}) and (\ref{errornewnew}). 

\end{proof}

To proceed to the computation of the main term (\ref{main}), we first show that it suffices to limit ourselves to the region where $u_1+\cdots+u_{k-1} \leq 1-\frac{1}{\log x}.$ 
Otherwise, if $u_1+\cdots+u_{k-1} > 1-\frac{1}{\log x},$ then we may assume without loss of generality that $u_{k-1}\geq \frac{1}{2k}.$ We will now show that when replacing $u_{k-1}$ by $u_{k-1}-\frac{1}{\log x},$ both the right-hand sides of (\ref{main}) and (\ref{fundamental}) are changed by a negligible amount.
Arguing similarly as before, we have
\begin{align} \label{replace}
\sum_{d_1 \leq x^{u_1}}\cdots\sum_{d_{k-2}\leq x^{u_{k-2}}}&
\sum_{x^{u_{k-1}-\frac{1}{\log x}}\leq d_{k-1} \leq x^{u_{k-1}}}\sum_{d_k\leq x/(d_1\cdots d_{k-1})}
\frac{1}{\tau_k(d_1\cdots d_{k})} \nonumber\\
\leq & \sum_{x^{u_{k-1}-\frac{1}{\log x}}\leq d_{k-1} \leq x^{u_{k-1}}}
\sum_{m\leq x/d_{k-1}}\frac{\tau_{k-1}(m)}{\tau_{k}(d_{k-1}m)}.
\end{align}
If $u_{k-1}\leq \frac{1}{2},$ then using \cite[Theorem~14.2]{koukoulopoulos2020distribution}, this is bounded by
\begin{align*}
\ll x^{u_{k-1}}\sum_{m\leq x^{1+\frac{1}{\log x}-u_{k-1}} }\frac{\tau_{k-1}(m)}{\tau_{k}(m)}
\ll \frac{x}{(\log x)^{\frac{1}{k}}}.
\end{align*}
Otherwise, again it follows from the observation (\ref{observation}) that (\ref{replace}) is
\begin{align*}
\ll x^{1-u_{k-1}}\sum_{d_{k-1}\leq x^{u_{k-1}}}\frac{\tau_{k-1}(d_{k-1})}{\tau_{k}(d_{k-1})}
\ll \frac{x}{(\log x)^{\frac{1}{k}}}.
\end{align*}

On the other hand, by making the change of variables $t_j=(1-t_{k-1})s_j$ for $j=1,\ldots, k-2,$ we have
\begin{align} \label{repetition}
x\int_{u_{k-1}-\frac{1}{\log x}}^{u_{k-1}}&t_{k-1}^{\frac{1}{k}-1}\left(\int_0^{u_1}\cdots\int_0^{u_{k-2}}
t_1^{\frac{1}{k}-1}\cdots t_{k-2}^{\frac{1}{k}-1}(1-t_1-\cdots-t_{k-1})^{\frac{1}{k}-1}
dt_1\cdots dt_{k-2}\right)dt_{k-1} \nonumber\\
\leq &  x  \int_{u_{k-1}-\frac{1}{\log x}}^{u_{k-1}}t_{k-1}^{\frac{1}{k}-1}(1-t_{k-1})^{-\frac{1}{k}}dt_{k-1} \nonumber\\
&\cdot\int_0^{\frac{u_1}{1-u_{k-1}}}\cdots\int_0^{\frac{u_{k-2}}{1-u_{k-1}}}
s_1^{\frac{1}{k}-1}\cdots s_{k-2}^{\frac{1}{k}-1}(1-s_1-\cdots-s_{k-2})^{\frac{1}{k}-1}
ds_1\cdots ds_{k-2} \nonumber\\
\ll &  x  \int_{0}^{\frac{1}{\log x}}t_{k-1}^{\frac{1}{k}-1}(1-t_{k-1})^{-\frac{1}{k}}dt_{k-1}
\ll \frac{x}{(\log x)^{\frac{1}{k}}}.
\end{align}
Therefore, 
we can always assume $u_1+\cdots+u_{k-1} \leq 1-\frac{1}{\log x}.$ 
Arguing similarly, we can further limit ourselves to the smaller region where $u_1,\dots,u_{k-1} \geq \frac{1}{\log x}$ as well.

In order to apply Lemma \ref{mainlemma}, we express the main term (\ref{main}) as 
\begin{align} \label{1to3}
\sum_{3 \leq d_1 \leq x^{u_1}}\cdots&\sum_{3 \leq d_{k-1}\leq x^{u_{k-1}}}\sum_{3 \leq d_k \leq x/(d_1\cdots d_{k-1})}\frac{1}{\tau_k(d_1\cdots d_{k})} \nonumber\\
&+O\left(\sum_{j=1}^{k-1}\sum_{d_j=1,2}
\underset{\substack{d_i \leq x^{u_i}\\i=1,\dots,k-1, i \neq j}}{\sum\cdots\sum}
\sum_{d_k \leq x/(d_1\cdots d_{k-1})} 
\frac{1}{\tau_k(d_1\cdots d_{k})}
\right) \nonumber \\
&+O\left(\sum_{d_k=1,2}\sum_{d_1 \leq x^{u_1}}\cdots\sum_{d_{k-1}\leq x^{u_{k-1}}}
\frac{1}{\tau_k(d_1\cdots d_{k})}
\right).
\end{align}
For $j=1,\dots,k-1,$ again it follows from \cite[Theorem~14.2]{koukoulopoulos2020distribution} that
\begin{align*}
\sum_{d_j=1,2} 
\underset{\substack{d_i \leq x^{u_i}\\i=1,\dots,k-1, i \neq j}}{\sum\cdots\sum}
\sum_{d_k \leq x/(d_1\cdots d_{k-1})} 
\frac{1}{\tau_k(d_1\cdots d_{k})}
&\leq  \sum_{m \leq x}\frac{\tau_{k-1}(m)}{\tau_k(m)} 
+\sum_{m \leq x/2} \frac{\tau_{k-1}(m)}{\tau_{k-1}(2m)}
\nonumber\\
&\ll \frac{x}{(\log x)^{\frac{1}{k}}},
\end{align*}
and similarly
\begin{align*} 
\sum_{d_{k}=1,2}\sum_{d_1 \leq x^{u_1}}\cdots\sum_{d_{k-1}\leq x^{u_{k-1}}}
\frac{1}{\tau_k(d_1\cdots d_{k})}
& \leq  \sum_{m \leq x}\frac{\tau_{k-1}(m)}{\tau_k(m)} 
+\sum_{m \leq x/2} \frac{\tau_{k-1}(m)}{\tau_{k-1}(2m)} \nonumber\\
&\ll \frac{x}{(\log x)^{\frac{1}{k}}}.
\end{align*}
By partial summation (or more precisely multiple Riemann-Stieltjes integration) and Lemma \ref{mainlemma}, the main term of (\ref{1to3}) is
\begin{align} 
&\int_{e}^{x^{u_1}}\cdots\int_{e}^{x^{u_{k-1}}}\int_{e}^{\frac{x}{x_1\cdots x_{k-1}}}\frac{1}{(\log x_1)^2}\cdots\frac{1}{(\log x_k)^2}d S(x_1,\dots,x_k) \nonumber\\
=&\frac{1}{\Gamma\left(\frac{1}{k} \right)^k}\int_{e}^{x^{u_1}}\cdots\int_{e}^{x^{u_{k-1}}}\int_{e}^{\frac{x}{x_1\cdots x_{k-1}}}(\log x_1)^{\frac{1}{k}-1}\cdots(\log x_k)^{\frac{1}{k}-1} dx_1\cdots dx_k \nonumber
\\&+\int_{e}^{x^{u_1}}\cdots\int_{e}^{x^{u_{k-1}}}\int_{e}^{\frac{x}{x_1\cdots x_{k-1}}}\frac{1}{(\log x_1)^2}\cdots\frac{1}{(\log x_k)^2}d R(x_1,\dots,x_k)\nonumber\\ 
=:& I_1 + I_2. \nonumber
\end{align}

Finally, it remains to compute the integrals $I_1$ and $I_2.$

\begin{lemma} \label{lemmaI_1}
The first integral $I_1$ equals
\begin{align*}
F(u_1,\ldots,u_{k-1})x+O\left(\frac{x}{(\log x)^{\frac{1}{k}}} \right).
\end{align*}
\end{lemma}
\begin{proof}
Making the change of variables $x_j=x^{t_j}$ for $j=1,\ldots,k-1$, the integral $I_1$ becomes
\begin{align*}
\frac{\log x}{\Gamma\left(\frac{1}{k} \right)^k}
\int_{{\frac{1}{\log x}}}^{u_1}\cdots\int_{{\frac{1}{\log x}}}^{u_{k-1}}\int_{{\frac{1}{\log x}}}^{1-t_1-\cdots -t_{k-1}}
t_1^{\frac{1}{k}-1}\cdots t_k^{\frac{1}{k}-1}x^{t_1+\cdots+t_k}dt_1\cdots dt_k.
\end{align*}
Integrating by parts with respect to $t_k$ gives
\begin{align} \label{magic}
\int_{\frac{1}{\log x}}^{1-t_1-\cdots-t_{k-1}}t_k^{\frac{1}{k}-1}x^{t_k}dt_k
= &(1-t_1-\cdots-t_{k-1})^{\frac{1}{k}-1}\frac{x^{1-t_1-\cdots-t_{k-1}}}{\log x}
-\frac{e}{(\log x)^{\frac{1}{k}}} \nonumber\\
&+\left(1-\frac{1}{k}\right)\frac{1}{\log x}\int_{\frac{1}{\log x}}^{1-t_1-\cdots t_{k-1}}t_k^{\frac{1}{k}-2}x^{t_k}dt_k.
\end{align}
Therefore, the contribution of the first term of
(\ref{magic}) to the integral $I_1$ is
\begin{align*} 
\frac{x}{\Gamma\left(\frac{1}{k} \right)^k}
\int_{{\frac{1}{\log x}}}^{u_1}\cdots\int_{{\frac{1}{\log x}}}^{u_{k-1}}
t_1^{\frac{1}{k}-1}\cdots t_{k-1}^{\frac{1}{k}-1}(1-t_1-\cdots t_{k-1})^{\frac{1}{k}-1}dt_1\cdots dt_{k-1}.
\end{align*}
Note that 
\begin{equation}
 \begin{gathered}[b] \label{eq:difference}
F(u_1,\ldots,u_{k-1})x
=\frac{x}{\Gamma\left(\frac{1}{k} \right)^k}
\int_{{\frac{1}{\log x}}}^{u_1}\cdots\int_{{\frac{1}{\log x}}}^{u_{k-1}}
t_1^{\frac{1}{k}-1}\cdots t_{k-1}^{\frac{1}{k}-1}(1-t_1-\cdots t_{k-1})^{\frac{1}{k}-1}dt_1\cdots dt_{k-1} 
\\
+O\left(x\sum_{j=1}^{k-1}\int_0^{\frac{1}{\log x}}\underset{\substack{0 \leq t_i \leq u_i\\i \neq j}}{\int\cdots\int}t_1^{\frac{1}{k}-1}\cdots t_{k-1}^{\frac{1}{k}-1}(1-t_1-\cdots t_{k-1})^{\frac{1}{k}-1}dt_1\cdots dt_{k-1} \right).
 \end{gathered}
\end{equation}
Without loss of generality, it suffices to bound the term where $j=k-1.$ Similar to (\ref{repetition}), we have 
\begin{align} \label{idk11}
x\int_0^{\frac{1}{\log x}}t_{k-1}^{\frac{1}{k}-1}&\left(\int_0^{u_1}\cdots\int_0^{u_{k-2}}
t_1^{\frac{1}{k}-1}\cdots t_{k-2}^{\frac{1}{k}-1}(1-t_1-\cdots-t_{k-1})^{\frac{1}{k}-1}
dt_1\cdots dt_{k-2}\right)dt_{k-1} \nonumber \\
&\ll x\int_0^{\frac{1}{\log x}}t_{k-1}^{\frac{1}{k}-1}(1-t_{k-1})^{-\frac{1}{k}}dt_{k-1} 
\ll \frac{x}{(\log x)^{\frac{1}{k}}}.
\end{align}

On the other hand, the contribution of the second term of (\ref{magic}) to the integral $I_1$ is
\begin{align*} 
\ll (\log x)^{1-\frac{1}{k}}\prod_{j=1}^{k-1}
\int_{\frac{1}{\log x}}^{u_j}t_j^{\frac{1}{k}-1}x^{t_j}dt_j 
&\ll (\log x)^{1-\frac{1}{k}} 
\prod_{j=1}^{k-1} u_j^{\frac{1}{k}-1}\frac{x^{u_j}}{\log x}.
\end{align*}
If $u_j > \frac{1}{2k}$ for some $j=1,\ldots, k-1,$ then 
this is 
\begin{align} \label{idk21}
\ll x(\log x)^{-(k-1)/k} \leq \frac{x}{(\log x)^{\frac{1}{k}}}.
\end{align}
as $k \geq 2$. Otherwise, the contribution is 
\begin{align*}
\ll x^{u_1+\cdots+u_{k-1}}(\log x)^{1-\frac{1}{k}} \ll x^{1/2}.
\end{align*}

We also have
\begin{align*}
\int_{\frac{1}{\log x}}^{1-t_1-\cdots t_{k-1}}t_k^{\frac{1}{k}-2}x^{t_k}dt_k &\ll (1-t_1-\cdots-t_{k-1})^{\frac{1}{k}-2}\frac{x^{1-t_1-\cdots-t_{k-1}}}{\log x} 
\end{align*}
so that the contribution of the last term of (\ref{magic}) to the integral $I_1$ is 
\begin{align*}
\ll \frac{x}{\log x}&\int_{\frac{1}{\log x}}^{u_1}\cdots\int_{\frac{1}{\log x}}^{u_{k-1}}t_1^{\frac{1}{k}-1}\cdots
t_{k-1}^{\frac{1}{k}-1}(1-t_1-\cdots-t_{k-1})^{\frac{1}{k}-2}dt_1\cdots dt_{k-1}.
\end{align*}
Making the change of variables $s=1-t_1-\cdots-t_{k-1},$ this is bounded by
\begin{align}  \label{idkx1}
\frac{x}{\log x} \int_{\frac{1}{\log x}}^{1-\frac{k-1}{\log x}}s^{\frac{1}{k}-2} 
&\left(\int_{\frac{1}{\log x}}^{u_1}\cdots\int_{\frac{1}{\log x}}^{u_{k-2}}
t_1^{\frac{1}{k}-1}\cdots t_{k-2}^{\frac{1}{k}-1}(1-s-t_1-\cdots-t_{k-2})^{\frac{1}{k}-1}dt_1\cdots dt_{k-2}\right)ds.
\end{align}
Similar to (\ref{repetition}), the integral in the parentheses is $\ll (1-s)^{-\frac{1}{k}},$ and so (\ref{idkx1}) is 
\begin{align} \label{idk31}
\ll \frac{x}{\log x} \int_{\frac{1}{\log x}}^{1-\frac{k-1}{\log x}}s^{\frac{1}{k}-2}(1-s)^{-\frac{1}{k}}ds \ll \frac{x}{(\log x)^{\frac{1}{k}}}.
\end{align}
Collecting the main term of (\ref{eq:difference}) and the error terms (\ref{idk11}), (\ref{idk21}) and (\ref{idk31}), the lemma follows.
\end{proof}

\begin{lemma} \label{lemmaI_2}
The second integral $I_2$ is
\begin{align*}
\ll \frac{x}{(\log x)^{\frac{1}{k}}}.
\end{align*}
\end{lemma}
\begin{proof} 
The integral $I_2$ is bounded by
\begin{align*}
\underset{\substack{l_j \leq  u_j \log x  , j=1,\ldots, k-1\\l_{k} \leq  \log x-l_1-\cdots-l_{k-1} }}{\sum \cdots \sum} I_{2}^{(\boldsymbol{l})},
\end{align*}
where 
\begin{align*}
I_{2}^{(\boldsymbol{l})}:=\underset{x_j \in [e^{l_j}, e^{l_j+1}], j=1,\dots,k}{\int \cdots \int}
\frac{1}{(\log x_1)^2}\cdots\frac{1}{(\log x_k)^2}dR(x_1,\ldots,x_k).
\end{align*}
By integration by parts, for each $\boldsymbol{l},$ the integral $I_{2}^{(\boldsymbol{l})}$ is bounded by
\begin{align*}
\sum_{J \subseteq [k]}
\underset{x_j \in [e^{l_j}, e^{l_j+1}], j\in J}{\int \cdots \int} 
\left(\sum_{x_j \in \{e^{l_j}, e^{l_j+1}\}, j \notin J} |R(x_1,\ldots, x_k)|\right)
\prod_{j \notin J}\frac{1}{{l_j}^2}\prod_{j \in J}\left|d\left(\frac{1}{(\log x_j)^2} \right) \right|
=:\sum_{J \subseteq [k]} I_{2}^{(\boldsymbol{l};J)}.
\end{align*}
For each subset $J \subseteq [k],$ the integral $I_{2}^{(\boldsymbol{l};J)}$ is
\begin{align*}
\ll  \left(\max_{x_j \in [e^{l_j}, e^{l_j+1}], j=1,\ldots,k}|R(x_1,\ldots,x_k)|\right)\underset{x_j \in [e^{l_j}, e^{l_j+1}], j=1,\ldots,k}{\int \cdots \int} \prod_{j \notin J}\frac{1}{x_j (\log x_j)^2}\prod_{j \in J} & \frac{1}{x_j (\log x_j)^3}\nonumber\\
\cdot & dx_1 \cdots dx_k.
\end{align*}
 Applying Lemma \ref{mainlemma}, this is
\begin{align*}
\ll \sum_{i=1}^k\underset{x_j \in [e^{l_j}, e^{l_j+1}], j=1,\ldots,k}{\int \cdots \int} \frac{1}{\log x_i} \prod_{j \notin J}\frac{1}{(\log x_j)^{1-\frac{1}{k}}} \prod_{j \in J} \frac{1}{(\log x_j)^{2-\frac{1}{k}}} dx_1\cdots dx_k.
\end{align*}
Summing over every $\boldsymbol{l},$ we have
\begin{align*}
\underset{\substack{l_j \leq  u_j \log x  , j=1,\ldots, k-1\\l_k \leq  \log x -l_1-\cdots-l_{k-1}}}{\sum \cdots \sum} I_2^{(\boldsymbol{l};J)}
\ll \sum_{i=1}^k \int_{e}^{x^{u_1}e}\cdots \int_{e}^{x^{u_{k-1}}e}\int_{e}^{\frac{e^kx}{x_1\cdots x_{k-1}}} &\frac{1}{\log x_i} \prod_{j \notin J}\frac{1}{(\log x_j)^{1-\frac{1}{k}}} \\
\cdot & \prod_{j \in J} \frac{1}{(\log x_j)^{2-\frac{1}{k}}} 
dx_1\cdots dx_k.
\end{align*}

To avoid repetitions, we only bound the contribution of the term where $i=k$ here.
Making the change of variables $x_j=x^{t_j}$ for $j=1,\ldots,k-1$, it becomes
\begin{align*}
\frac{1}{(\log x)^{|J|}} \int_{\frac{1}{\log x}}^{u_1+\frac{1}{\log x}} \cdots \int_{\frac{1}{\log x}}^{u_{k-1}+\frac{1}{\log x}}\int_{\frac{1}{\log x}}^{1+\frac{k}{\log x}-t_1-\cdots-t_{k-1}}t_k^{-1} \prod_{j \notin J}t_j^{\frac{1}{k}-1}\prod_{j \in J}t_j^{\frac{1}{k}-2} x^{t_1+\cdots+t_k}
 dt_1\cdots dt_k,
\end{align*}
which is
\begin{align} \label{almostthere}
\ll \int_{\frac{1}{\log x}}^{u_1+\frac{1}{\log x}} \cdots \int_{\frac{1}{\log x}}^{u_{k-1}+\frac{1}{\log x}}\int_{\frac{1}{\log x}}^{1+\frac{k}{\log x}-t_1-\cdots-t_{k-1}}&t_1^{\frac{1}{k}-1}\cdots t_{k-1}^{\frac{1}{k}-1}t_k^{\frac{1}{k}-2}
x^{t_1+\cdots+t_k}
dt_1\cdots dt_k.
\end{align}
Integrating by parts with respect to $t_k$ gives
\begin{align*}
\int_{\frac{1}{\log x}}^{1+\frac{k}{\log x}-t_1-\cdots-t_{k-1}} t_k^{\frac{1}{k}-2}x^{t_k}dt_k 
\ll \left(1+\frac{k}{\log x}-t_1-\cdots-t_{k-1} \right)^{\frac{1}{k}-2}\frac{x^{1-t_1-\cdots-t_{k-1}}}{\log x}.
\end{align*}
Therefore, the expression (\ref{almostthere}) is
\begin{align*}
\ll \frac{x}{\log x} \int_{\frac{1}{\log x}}^{u_1+\frac{1}{\log x}} \cdots \int_{\frac{1}{\log x}}^{u_{k-1}+\frac{1}{\log x}} t_1^{\frac{1}{k}-1}\cdots t_{k-1}^{\frac{1}{k}-1}
\left(1+\frac{k}{\log x}-t_1-\cdots-t_{k-1} \right)^{\frac{1}{k}-2} dt_1\cdots d_{k-1}.
\end{align*}
Similar to (\ref{repetition}), this is
\begin{align*}
\ll \frac{x}{\log x}\int_{\frac{2}{\log x}}^{1+\frac{1}{\log x}}
s^{\frac{1}{k}-2}
\left( 1+\frac{k}{\log x}-s \right)^{-\frac{1}{k}}ds
\ll \frac{x}{(\log x)^{\frac{1}{k}}}.
\end{align*}
Finally, the lemma follows from summing over every subset $J \subseteq [k].$
\end{proof}


\section{Proof of Theorem \ref{thm:dirichlet2}}
We first define the function field analogue of the multiple Dirichlet series $\mathcal{D}(s_1,\dots,s_k).$



\begin{definition}
For $(s_1,\dots,s_k) \in \Omega$, the multiple Dirichlet series $\mathcal{D}_{\mathbb{F}_q[x]}(s_1,\dots,s_k)$ is defined as
\begin{align*} 
\underset{F_1,\dots,F_k \in \mathcal{M}_q}{\sum\cdots
\sum}\frac{\tau_k(F_1\cdots F_k)^{-1}}{q^{s_1\deg F_1+\cdots+s_k \deg F_k}}.
\end{align*}
\end{definition}
Having  the multiple Dirichlet series $\mathcal{D}_{\mathbb{F}_q[x]}(s_1,\dots,s_k)$ in hand, it is now clear that we can follow exactly the same steps as before. Moreover, since the function field zeta function $\zeta_{\mathbb{F}_q}(s)$ never vanishes (see \cite[Chapter~2]{rosen2002number}), some of the computations above can be simplified considerably. To avoid repetitions, the complete proof is omitted here.

\section{Proof of Theorem \ref{thm:dirichlet3}}\label{section5} 
One may repeat the same argument using a multiple Dirichlet series but a more direct and elementary proof is presented here.
We begin with the combinatorial analogue of the mean of divisor functions.
\begin{lemma} \label{6.1} Let $\alpha \in \mathbb{C}\setminus \mathbb{Z}_{\leq 0}$ and $n\in \mathbb{Z}_{\geq 0}.$ Then we have
\begin{align} \label{eq:lem6.11}
\frac{1}{n!}\sum_{\sigma \in S_n}\tau_{\alpha}(\sigma)={n+\alpha-1 \choose n}.
\end{align}
Moreover, we have the estimate
\begin{align} \label{eq:lem6.12}
{n+\alpha-1 \choose n}=
\begin{cases} 
1 &\mbox{if } n = 0, \\
\frac{1}{\Gamma \left(\alpha \right)}n^{\alpha-1}\left(1+O_{\alpha}\left(\frac{1}{n}\right) \right) & \mbox{if } n\geq1.
\end{cases}
\end{align}
\end{lemma}
\begin{proof}
Although (\ref{eq:lem6.11}) is fairly standard, say for instance one may 
apply \cite[Corollary~5.1.9]{MR1676282} with $f \equiv \alpha$, we provide a detailed proof here for the sake of completeness.
Adopting the notations in Section \ref{section2}, we write
\begin{align*}
\sum_{\sigma \in S_n}\tau_{\alpha}(\sigma)
&=\sum_{k=0}^{n}\left[{n\atop k}\right]\alpha^k.
\end{align*}

Now any permutation $\sigma \in S_n$ with $k$ disjoint cycles can be constructed by the following procedure.
To begin with, there are $(n-1)(n-2)\cdots(n-i_1+1)$ ways of choosing $i_1-1$ distinct integers from $[n]\setminus\{1\}$ to form a cycle $C_1$ with length $i_1$ containing $1.$ Then, fix any integer $m \in [n]$ not contained in the cycle $C_1$. Similarly, there are $(n-i_1-1)\cdots(n-i_1-i_2+1)$ ways of choosing $i_2-1$ integers from $[n]\setminus C_1$ to form another cycle $C_2$ with length $i_2$ containing $m.$  Repeating the same procedure until $i_1+\cdots+i_k=n$, we arrive at a permutation with $k$ disjoint cycles.

Therefore, the expression (\ref{eq:lem6.11}) follows from the explicit formula
\begin{align*}
\left[{n\atop k}\right]=\underset{i_1+\cdots+i_k=n}{\sum_{i_1=1}^n\cdots\sum_{i_k=1}^n}\frac{(n-1)!}{\prod_{j=1}^{k}(n-i_1-\cdots-i_j)},
\end{align*}
which can be seen as the coefficient of ${\alpha}^k$ in the falling factorial
\begin{align*}
(n+\alpha-1)(n+\alpha-2)\cdots(\alpha+1)\alpha.
\end{align*}

On the other hand, to prove (\ref{eq:lem6.12}), we express the binomial coefficient as a ratio of gamma functions, followed by the application of Stirling's formula.
\end{proof}

Similar to Theorem \ref{thm:dirichlet}, without loss of generality we shall assume 
$u_1,\ldots u_{k-1}, 1-u_1-\cdots-u_{k-1} \geq \frac{1}{n}.$ 
Interchanging the order of summation, we have 
\begin{align} \label{b4kickoff}
\frac{1}{n!}\sum_{\sigma \in S_n}\frac{1}{\tau_k(\sigma)}
\underset{{\substack{[n]=A_1 \sqcup \cdots \sqcup A_k\\ \sigma(A_i)=A_i, i=1,\dots,k \\ 0 \leq |A_i|  \leq nu_i,  i=1,\dots,k-1}}}{\sum\cdots\sum}1
&=\frac{1}{n!}\underset{\substack{0 \leq m_i \leq nu_i \\ i=1,\dots,k-1}}{\sum\cdots\sum}{n \choose m_1,\dots,m_k}\prod_{i=1}^k
\left(\sum_{\sigma \in S_{m_i}}\frac{1}{\tau_k(\sigma)} \right) \nonumber\\
&=\underset{\substack{0 \leq m_i \leq nu_i \\ i=1,\dots,k-1}}{\sum\cdots\sum}
\prod_{i=1}^k \left(\frac{1}{m_i!}\sum_{\sigma \in S_{m_i}}\frac{1}{\tau_k(\sigma)} \right),
\end{align}
where $m_k:=n-m_1-\cdots -m_{k-1}.$\\
Note that $\tau_k(\sigma)^{-1}=\tau_{1/k}(\sigma).$
 Applying (\ref{eq:lem6.11}) from Lemma \ref{6.1}, the expression (\ref{b4kickoff}) equals
\begin{align} \label{kickoff}
\underset{\substack{0 \leq m_i \leq nu_i \\ i=1,\dots,k-1}}{\sum\cdots\sum}\prod_{i=1}^k
{m_i+\frac{1}{k}-1 \choose m_i}.
\end{align}




Let $I \subseteq [k-1]$ be a nonempty subset. Then using (\ref{eq:lem6.12}) from Lemma \ref{6.1}, the contribution of $m_i=0$ to (\ref{kickoff}) for $i\in I$ is 
\begin{align} 
\ll \underset{\substack{1 \leq m_i \leq nu_i\\ i \notin I}}{\sum\cdots\sum}
\left(\prod_{i \notin I}m_i^{\frac{1}{k}-1}\right)m_k^{\frac{1}{k}-1}
= n^{-\frac{|I|}{k}}\underset{\substack{1 \leq m_i \leq nu_i \\ i \notin I}}{\sum\cdots\sum}
&\left(\prod_{i \notin I}\left(\frac{m_i}{n}\right)^{\frac{1}{k}-1}\right)
\nonumber \\
\cdot & \left(1-\frac{m_1}{n}-\cdots-\frac{m_{k-1}}{n}\right)^{\frac{1}{k}-1}
n^{-(k-|I|)}, \nonumber
\end{align}
which is
\begin{align} \label{iI}
\ll n^{-\frac{|I|}{k}} \underset{\substack{0 \leq t_i \leq u_i \\ i \notin I}}{\int\cdots\int}\left(\prod_{i \notin I}t_i^{\frac{1}{k}-1}\right)
\left(1-\sum_{i \notin I}t_i\right)^{\frac{1}{k}-1}\prod_{i \notin I}dt_i
\ll  n^{-\frac{|I|}{k}}.\
\end{align}

Also, the contribution of $m_j>nu_j-1$ for some $j=1,\dots,k-1$ to (\ref{kickoff}) given that $m_1,\dots,m_k\geq 1$ is 
\begin{align} 
\ll \sum_{j=1}^{k-1}{(nu_j)}^{\frac{1}{k}-1}\underset{\substack{1 \leq m_i \leq nu_i\\ i\neq j}}{\sum\cdots\sum}
\prod_{\substack{i=1\\i\neq j}}^k m_i^{\frac{1}{k}-1} 
=n^{-\frac{1}{k}} 
\sum_{j=1}^{k-1}{(nu_j)}^{\frac{1}{k}-1}\underset{\substack{1 \leq m_i \leq nu_i\\ i\neq j}}{\sum\cdots\sum}
\prod_{\substack{i=1\\i\neq j}}^k\left(\frac{m_i}{n}\right)^{\frac{1}{k}-1}. \nonumber
\end{align}
Since $u_j \geq \frac{1}{n}$ for $j=1,\dots,k-1,$ this is
\begin{align} \label{nu_j-1}
&\ll n^{-\frac{1}{k}} \sum_{j=1}^{k-1} (nu_j)^{\frac{1}{k}-1}\underset{\substack{0 \leq t_i \leq u_i, i=1,\dots,k-1\\i \neq j}}{\int\cdots\int}
\prod_{\substack{i=1\\i \neq j}}^{k-1}t_i^{\frac{1}{k}-1}
(1-t_1-\cdots-t_{k-1})^{\frac{1}{k}-1}dt_1\cdots dt_{k-1} \nonumber\\
&\ll n^{-\frac{1}{k}} \sum_{j=1}^{k-1} (nu_j)^{\frac{1}{k}-1}\left(1-u_j\right)^{-\frac{1}{k}} 
\ll n^{-\frac{1}{k}}.
\end{align}
Collecting the error terms 
(\ref{iI}) and (\ref{nu_j-1}), the expression (\ref{kickoff}) equals
\begin{align} \label{generic}
\underset{\substack{1 \leq m_i \leq nu_i-1 \\ i=1,\dots,k-1}}{\sum\cdots\sum}\prod_{i=1}^k{m_i+\frac{1}{k}-1\choose m_i}+O(n^{-\frac{1}{k}}).
\end{align}
 Applying (\ref{eq:lem6.12}) from Lemma \ref{6.1},  the main term of (\ref{generic}) is the Riemann sum
\begin{align}\label{riemsum}
\frac{1}{\Gamma \left(\frac{1}{k} \right)^k}\underset{\substack{1 \leq m_i \leq nu_i-1 \\ i=1,\dots,k-1}}{\sum\cdots\sum}
\prod_{i=1}^k \left(\frac{m_i}{n}\right)^{\frac{1}{k}-1}\frac{1}{n^{k-1}},
\end{align}
with an error term 
\begin{align} \label{funderror}
\ll \frac{1}{n}\sum_{j=1}^{k}\left(\frac{m_j}{n}\right)^{\frac{1}{k}-2}\underset{\substack{1 \leq m_i \leq nu_i-1\\ i=1,\dots,k-1}}{\sum\cdots\sum}
\prod_{\substack{i=1\\i \neq j}}^k\left(\frac{m_i}{n}\right)^{\frac{1}{k}-1}\frac{1}{n^{k-1}}.
\end{align}

Let us first bound the error term (\ref{funderror}). For each $j=1,\dots,k-1,$ we have
\begin{align} 
\frac{1}{n}&\sum_{1 \leq m_j \leq nu_j-1}\left(\frac{m_j}{n} \right)^{\frac{1}{k}-2}\underset{\substack{1 \leq m_i \leq nu_i-1 \\i=1,\dots,k-1, i\neq j}}{\sum\cdots\sum}\prod_{\substack{i=1\\i \neq j}}^{k-1}\left(\frac{m_i}{n} \right)^{\frac{1}{k}-1}\frac{1}{n^{k-1}} \nonumber\\
&\ll \frac{1}{n} \int_{\frac{1}{n}}^{1-\frac{1}{n}} t_j^{\frac{1}{k}-2}
\underset{\substack{0 \leq t_i \leq u_i\\i=1,\dots,k-1, i \neq j}}{\int\cdots\int}
\prod_{\substack{i=1\\i \neq j}}^{k-1}t_i^{\frac{1}{k}-1}
(1-t_1-\cdots-t_{k-1})^{\frac{1}{k}-1}dt_1\cdots dt_{k-1} \nonumber
\\
&\ll \frac{1}{n}\int_{\frac{1}{n}}^{1-\frac{1}{n}}
t_j^{\frac{1}{k}-2}(1-t_j)^{-\frac{1}{k}}dt_j 
\ll n^{-\frac{1}{k}}. \nonumber
\end{align}
Arguing similarly for $j=k,$ we also have
\begin{align} 
\frac{1}{n}\underset{\substack{1 \leq m_i \leq nu_i-1 \\ i=1,\dots,k-1}}{\sum\cdots\sum}\prod_{i=1}^{k-1}\left(\frac{m_i}{n} \right)^{\frac{1}{k}-1}\left(\frac{m_k}{n} \right)^{\frac{1}{k}-2}\frac{1}{n^{k-1}} \ll n^{-\frac{1}{k}}.
\nonumber
\end{align}
Therefore, the error term (\ref{funderror}) is $\ll n^{-\frac{1}{k}}$ and we are left with the main term (\ref{riemsum}).

The distribution function $F(u_1,\ldots,u_{k-1})$ equals
\begin{align} \label{fullintegral}
\frac{1}{\Gamma \left(\frac{1}{k} \right)^k}&\underset{\substack{1 \leq m_i \leq nu_i-1 \\ i=1,\dots,k-1 }}{\sum\cdots\sum}\int_{\frac{m_1}{n}}^{\frac{m_1+1}{n}}\cdots\int_{\frac{m_{k-1}}{n}}^{\frac{m_{k-1}+1}{n}}t_1^{\frac{1}{k}-1}\cdots t_{k-1}^{\frac{1}{k}-1}(1-t_1-\cdots-t_{k-1})^{\frac{1}{k}-1}dt_1\cdots dt_{k-1} \nonumber\\
&+O\left(\sum_{j=1}^{k-1}\int_{0}^{\frac{1}{n}}t_j^{\frac{1}{k}-1}\underset{\substack{0 \leq t_i \leq u_i\\ i\neq j}}{\int\cdots\int}\prod_{\substack{i=1\\i \neq j}}^{k-1}t_j^{\frac{1}{k}-1}(1-t_1-\cdots-t_{k-1})^{\frac{1}{k}-1}dt_1\cdots dt_{k-1} \right) \nonumber\\
&+O\left(\sum_{j=1}^{k-1}\int_{u_j-\frac{1}{n}}^{u_j}t_j^{\frac{1}{k}-1}\underset{\substack{0 \leq t_i \leq u_i \\ i\neq j}}{\int\cdots\int}\prod_{\substack{i=1\\i \neq j}}^{k-1}t_j^{\frac{1}{k}-1}(1-t_1-\cdots-t_{k-1})^{\frac{1}{k}-1}dt_1\cdots dt_{k-1} \right).
\end{align}
The first error term in (\ref{fullintegral}) is 
\begin{align} \label{idk1}
\ll \sum_{j=1}^{k-1}\int_{0}^{\frac{1}{n}}t_j^{\frac{1}{k}-1}(1-t_j)^{-\frac{1}{k}}dt_j
\ll n^{-\frac{1}{k}}.
\end{align}
The second error term in (\ref{fullintegral}) is 
\begin{align} \label{idk2}
\ll \sum_{j=1}^{k-1}\int_{u_j-\frac{1}{n}}^{u_j}t_j^{\frac{1}{k}-1}(1-t_j)^{-\frac{1}{k}}dt_j
&\leq \sum_{j=1}^{k-1}\int_{0}^{\frac{1}{n}}t_j^{\frac{1}{k}-1}(1-t_j)^{-\frac{1}{k}}dt_j \nonumber\\
&\ll n^{-\frac{1}{k}}.
\end{align}

By Taylor's theorem, for $(t_1,\dots,t_{k-1}) \in 
\left[\frac{m_1}{n}, \frac{m_1+1}{n} \right]\times \cdots \times [\frac{m_{k-1}}{n}, \frac{m_{k-1}+1}{n}],$
we have
\begin{align*} 
t_1^{\frac{1}{k}-1}\cdots t_{k-1}^{\frac{1}{k}-1}(1-t_1-\cdots-t_{k-1})^{\frac{1}{k}-1}
=\prod_{i=1}^k\left(\frac{m_i}{n}\right)^{\frac{1}{k}-1}
+O\left(\frac{1}{n}\sum_{j=1}^{k}\left(\frac{m_j}{n}\right)^{\frac{1}{k}-2}
\prod_{\substack{i=1\\i \neq j}}^k\left(\frac{m_i}{n}\right)^{\frac{1}{k}-1}\right).
\end{align*}
Using the approximation, we conclude from (\ref{fullintegral}), (\ref{idk1}) and (\ref{idk2}) that
\begin{align*}
F(u_1,\ldots,u_{k-1})=&\frac{1}{\Gamma \left(\frac{1}{k} \right)^k}\underset{\substack{1 \leq m_i \leq nu_i-1 \\ i=1,\dots,k-1}}{\sum\cdots\sum}
\prod_{i=1}^k \left(\frac{m_i}{n}\right)^{\frac{1}{k}-1}\frac{1}{n^{k-1}}
+O\left(n^{-\frac{1}{k}} \right)\\
&+O\left(\frac{1}{n}
\underset{\substack{1 \leq m_i \leq nu_i-1 \\ i=1,\dots,k-1}}{\sum\cdots\sum}
\sum_{j=1}^{k}
\left(\frac{m_j}{n}\right)^{\frac{1}{k}-2}
\prod_{\substack{i=1\\i \neq j}}^k\left(\frac{m_i}{n}\right)^{\frac{1}{k}-1}\frac{1}{n^{k-1}} \right),
\end{align*}
and the last error term here is exactly the same as (\ref{funderror}), which is again $\ll n^{-\frac{1}{k}}.$

\section{Factorization into $k$ parts in the general setting } \label{section6}

With a view to model Dirichlet distribution with arbitrary parameters, we further explore the factorization of integers into $k$ parts in the general setting using multiplicative functions of several variables defined below. 

\begin{definition} An arithmetic function of $k$ variables $F \colon \mathbb{N}^k \to \mathbb{C}$ is said to be multiplicative if it satisfies the condition 
$F(1,\ldots,1)=1$ and the functional equation
\begin{align*}
F(m_1n_1,\ldots,m_kn_k)=F(m_1,\ldots,m_k)F(n_1,\ldots,n_k)
\end{align*}
whenever $(m_1\cdots m_k, n_1\cdots n_k)=1,$ or equivalently, 
\begin{align*}
F(n_1,\ldots,n_k)=\prod_{p}F(p^{v_p(n_1)},\ldots,p^{v_p(n_k)}),
\end{align*}
where $v_p(n):=\max\{k \geq 0 \, : \, p^k | n\}.$
\end{definition}

\begin{remark}
Multiplicative functions of several variables, such as the ``GCD function'' and the ``LCM function'' are interesting for their own sake. See \cite{toth2014multiplicative} for further discussion.
\end{remark}

To adapt the proof of Theorem \ref{thm:dirichlet}, we 
consider the following class of 
multiplicative functions.
\begin{definition} \label{defclass1}
Let $\boldsymbol{\alpha}=(\alpha_1,\ldots,\alpha_k), \boldsymbol{\beta}=(\beta_1,\ldots,\beta_k), \boldsymbol{c}= (c_1,\ldots,c_k)
,\boldsymbol{\delta}=(\delta_1,\ldots,\delta_k)$
with $\alpha_j, \beta_j, c_j>0, \delta_j\geq 0$ for $j=1,\ldots,k.$ We denote by
$\mathcal{M}(\boldsymbol{\alpha};\boldsymbol{\beta},\boldsymbol{c}, \boldsymbol{\delta})$ the class of non-negative multiplicative functions of $k$ variables $F \colon \mathbb{N}^k \to \mathbb{C}$ satisfying the following conditions:
\begin{enumerate}[label=(\alph*)]
\item (divisor bound) for $j=1,\ldots,k,$ we have $|F(1,\ldots,\overbrace{n}^{\text{$j$-th}},\ldots 1)| \leq \tau_{\beta_j}(n),$ 
where 
\begin{align*}
\tau_{\beta}(n)
:=\prod_{p}{v_p(n)+\beta-1 \choose v_p(n)}
\end{align*}
is the generalized divisor function;
\item (analytic continuation) let $s=\sigma+it \in \mathbb{C}.$ For $j=1,\ldots,k$, the Dirichlet series 
\begin{align*}
\mathcal{P}_F(s;\boldsymbol{\alpha},j):=\sum_{p}\frac{F(1,\ldots,\overbrace{p}^{\text{$j$-th}},\ldots 1)-\alpha_j}{p^{s}}
\end{align*}
defined for $\sigma>1$ can be continued analytically to the domain where $\sigma>1-\frac{c_j}{\log (2+|t|)};$

\item (growth rate) for $j=1,\ldots,k,$ in the domain above we have the bound
\begin{align*}
\mathcal{P}_F(s;\boldsymbol{\alpha},j) \leq \delta_j\log (2+|t|).
\end{align*}
\end{enumerate}
\end{definition}
For instance, the multiplicative function 
$F(n_1,\ldots,n_k)=\tau_k(n_1\cdots n_k)^{-1}$ belongs to the class $\mathcal{M}\left(\boldsymbol{\frac{1}{k}};\boldsymbol{\frac{1}{k}},
\boldsymbol{1},
\boldsymbol{0}\right).$

Applying the Mellin transform to (higher derivatives of) the multiple Dirichlet series
\begin{align*}
\mathcal{D}_F(s_1,\ldots,s_k):=
\sum_{n_1=1}^{\infty}\cdots
\sum_{n_k=1}^{\infty}\frac{F(n_1,\ldots,n_k)}{n_1^{s_1}\cdots n_k^{s_k}}
\end{align*}
as before, one can prove the following generalization of Lemma \ref{mainlemma}.

\begin{lemma} \label{lemmageneral}
Given a multiplicative function of $k$ variables $ F \in \mathcal{M}(\boldsymbol{\alpha};\boldsymbol{\beta}, \boldsymbol{c}, \boldsymbol{\delta}).$ Let  $m \geq 2$ be an integer and $x_1,\dots,x_k \geq e.$ We denote by $S_F(x_1,\dots,x_k;m)$ the weighted sum
\begin{align*}
\sum_{d_1\leq x_1}\cdots\sum_{d_k\leq x_k}(\log d_1)^m \cdots (\log d_k)^m F(d_1,\ldots,d_k).
\end{align*}
Then there exists $m_0=m_0(\boldsymbol{\alpha},\boldsymbol{\beta}, \boldsymbol{c}, \boldsymbol{\delta})$ such that for any integer $m\geq m_0,$ we have 
\begin{align*}
S_F(x_1,\ldots,x_k;m) = \prod_{j=1}^k\frac{1}{\Gamma(\alpha_j)}\int_{1}^{x_j}
(\log y_j)^{\alpha_j+m-1} dy_j + R_F(x_1,\ldots,x_k;m)
\end{align*}
with
\begin{align*}
R_F(x_1,\dots,x_k;m) \ll x_1\cdots x_k\sum_{j=1}^k \left(\prod_{\substack{i=1\\i \neq j}}^k  (\log x_i)^{\alpha_i+m-1}\right) (\log x_j)^{\alpha_j+m-2}.
\end{align*}
\end{lemma}


To model the Dirichlet distribution by factorizing integers into $k$ parts, we consider the following class of pairs of multiplicative functions.
\begin{definition}
Let $\theta>0$ and $\boldsymbol{\alpha}$ be a positive $k$-tuple.
We denote by
$\mathcal{M}_{\theta}(\boldsymbol{\alpha})$ 
the class of pairs of multiplicative functions $(f; G) $ satisfying the following conditions:
\begin{enumerate}[label=(\alph*)]
\item for $n \geq 1,$ we have 
\begin{align*}
\sum_{n=d_1\cdots d_k}G(d_1,\ldots,d_k)>0;
\end{align*}
\item the multiplicative function $f$ belongs to the class $\mathcal{M}(\theta;\beta',c', \delta' )$ for some $\beta', c', \delta';$
\item the multiplicative function of $k$ variables
\begin{align*}
F(d_1,\ldots,d_k):=f(n)\cdot 
\frac{G(d_1,\ldots,d_k)}{\sum_{n=e_1\cdots e_k}G(e_1,\ldots,e_k)}
\end{align*}
belongs to the class $\mathcal{M}(\boldsymbol{\alpha};\boldsymbol{\beta}, \boldsymbol{c}, \boldsymbol{\delta})$ for some $\boldsymbol{\beta}, \boldsymbol{c}, \boldsymbol{\delta},$ where $n=d_1\cdots d_k.$
\end{enumerate}
\end{definition}

\begin{remark}
By definition, we must have $\theta=\alpha_1+\cdots+\alpha_k.$
\end{remark}

Then, applying Lemma \ref{lemmageneral} followed by partial summation as before, one can prove the following generalization of Theorem \ref{thm:dirichlet}.

\begin{theorem} \label{thmgeneral}
Let $(f;G)$ be a pair of multiplicative functions belonging to the class $\mathcal{M}_{\theta}(\boldsymbol{\alpha}).$ Then uniformly for $x \geq 2$ and $u_1,\ldots,u_{k-1} \geq 0$ satisfying $u_1+\cdots+u_{k-1}\leq 1,$ we have
\begin{align*}
\left(\sum_{m \leq x}f(m)\right)^{-1}&\sum_{n\leq x}f(n)
\left( \sum_{n=e_1\cdots e_k}G(e_1,\ldots,e_k) \right)^{-1}
\underset{n=d_1\cdots d_{k}}{\sum_{d_1 \leq n^{u_1}}\cdots
\sum_{d_{k-1} \leq n^{u_{k-1}}}\sum_{d_k \leq n}}
G(d_1,\ldots,d_k)
\\
=&F_{\boldsymbol{\alpha}}(u_1,\ldots, u_{k-1})
+O\left(\frac{1}{(\log x)^{\min \{1,\alpha_1,\ldots,\alpha_k\}}} \right).
\end{align*}
\end{theorem}

Finally, we conclude with the following generalization of Corollary \ref{cor1}.
\begin{corollary} \label{gencor}
Given a pair of multiplicative functions $(f;G)$ belonging to the class $\mathcal{M}_{\theta}(\boldsymbol{\alpha}).$
For $x\geq 1,$ let $n$ be a random integer chosen from $[1,x]$ with probability $\left(\sum_{m \leq x}f(m)\right)^{-1}f(n)$ and $(d_1,\ldots,d_k)$ be a random $k$-tuple chosen from the set of all possible factorization $\{(m_1,\ldots,m_k) \in \mathbb{N}^k\, : \, n=m_1\cdots m_k\}$ with probability $\left( \sum_{n=e_1\cdots e_k}G(e_1,\ldots,e_k) \right)^{-1}G(d_1,\ldots,d_k)$. Then we have the convergence in distribution
\[\left(\frac{\log d_1}{\log n}, \ldots, \frac{\log d_k}{\log n}\right) \xrightarrow[]{d}
\mathrm{Dir}\left(\alpha_1,\ldots,\alpha_k \right)\]
as $x \to \infty.$
\end{corollary}

\begin{remark}
See \cite{bareikis2017modeling}, 
\cite{bareikis2021bivariate} for the cases where $k=2,3$ respectively, in which $G(d_1,\ldots,d_k)$ takes the form $(f_1\ast \cdots \ast f_{k-1} \ast 1)(d_1\cdots d_k)$ for some multiplicative functions $f_1,\ldots,f_{k-1}:\mathbb{N} \to \mathbb{C}.$
\end{remark}


\begin{example}
For $k\geq 2,$ let $\theta, \lambda_1, \ldots, \lambda_k>0.$ We consider the pair of multiplicative functions
\begin{align*}
f(n)=\tau_{\theta}(n);\quad G(d_1,\ldots,d_k)
= \tau_{\lambda_1}(d_1) \cdots 
\tau_{\lambda_k}(d_k).
\end{align*}
Then the Dirichlet distribution of dimension $k$
\begin{align*}
\mathrm{Dir}\left(\frac{\theta\lambda_1}{\lambda_1+\cdots+\lambda_k},\ldots,\frac{\theta\lambda_k}{\lambda_1+\cdots+\lambda_k} \right)
\end{align*}
can be modelled in the sense of Corollary \ref{gencor}. In particular, when $\theta, \lambda_1,\ldots, \lambda_k=1,$ it reduces to Theorem \ref{thm:dirichlet}. 
\end{example}

\begin{example}
For $q\geq 3,$ let $\{a_1,\ldots,a_{\varphi(q)}\}$ be a reduced residue system$\Mod{q}$. We consider the pair of multiplicative functions
\begin{align*}
f(n)=
\begin{cases}
1 & \mbox{if $(n,q)=1$},\\
0 & \mbox{otherwise}
\end{cases}
;\quad 
G(d_1,\ldots,d_k)=
\begin{cases}
1 & \mbox{if $p|d_j$ implies $p \equiv a_j \Mod{q} $ for $j=1,\ldots,\varphi(q)$},\\
0 & \mbox{otherwise}.
\end{cases}
\end{align*}
Then the Dirichlet distribution of dimension $\varphi(q)$
\begin{align*}
\mathrm{Dir}\left(\frac{1}{\varphi(q)},\ldots,\frac{1}{\varphi(q)} \right)
\end{align*}
can be modelled in the sense of Corollary \ref{gencor}. In particular, when $q=4,$ it reduces to 
\cite[Exercise 6.2.22]{montgomery2007multiplicative}.
\end{example}

\begin{example}
For $k\geq 2,$ we consider the pair of multiplicative functions
\begin{align*}
f(n)=
\begin{cases}
1 & \mbox{if $n$ is a sum of two squares} ,\\
0 & \mbox{otherwise}
\end{cases}
;\quad G(d_1,\ldots,d_k)\equiv 1.
\end{align*}
Then the Dirichlet distribution of dimension $k$
\begin{align*}
\mathrm{Dir}\left(\frac{1}{2k},\ldots,\frac{1}{2k} \right)
\end{align*}
can be modelled in the sense of Corollary \ref{gencor}. In particular, when $k=2,$ it reduces to \cite[Theorem 2]{daoud2015distribution}.
\end{example}

\begin{example}
For $k\geq 2,$ we consider the pair of multiplicative functions
\begin{align*}
f(n)=
\begin{cases}
1 & \mbox{if $n$ is square-free} ,\\
0 & \mbox{otherwise}
\end{cases}
;\quad G(d_1,\ldots,d_k)\equiv 1.
\end{align*}
Then the Dirichlet distribution of dimension $k$
\begin{align*}
\mathrm{Dir}\left(\frac{1}{k},\ldots,\frac{1}{k} \right)
\end{align*}
can be modelled in the sense of Corollary \ref{gencor}. In particular, when $k=2,$ it reduces to \cite[Theorem 2]{MR3581662} with $y=x.$
\end{example}

\begin{example}
For $k\geq 2$, let $\mathcal{R}$ be a subset of $\{\{i,j\}\,:\, 1 \leq i \neq j \leq k\}.$ We consider the pair of multiplicative functions
\begin{align*}
f(n) \equiv 1;\quad 
G(d_1,\ldots,d_k)=
\begin{cases}
1 & \mbox{if $(d_i,d_j)=1$ whenever 
$\{i,j\}\notin \mathcal{R},$
}\\
0 & \mbox{otherwise.}
\end{cases}
\end{align*}
Then the Dirichlet distribution of dimension $k$
\begin{align*}
\mathrm{Dir}\left( \frac{1}{k},\ldots,\frac{1}{k} \right)
\end{align*}
can be modelled in the sense of Corollary \ref{gencor}. In particular, when $k=2^r$ for $r \geq 2,$ it reduces to \cite[Th\'eor\`eme 1.1]{de2016processus} with a suitable subset $\mathcal{R}$ via total decomposition sets (see \cite[Theorem 0.20]{MR1414678}), which is itself a generalization of \cite[Theorem 2.1]{bareikis2012cesaro} for $r=2.$
\end{example}

\begin{example}
For $k\geq 3$, we consider the pair of multiplicative functions
\begin{align*}
f(n) \equiv 1;\quad 
G(d_1,\ldots,d_k)=
\prod_{j=1}^{k-1}
\frac{1}{\tau(d_j\cdots d_k)}.
\end{align*}
Then the Dirichlet distribution of dimension $k$
\begin{align*}
\mathrm{Dir}\left(\frac{1}{2},\frac{1}{4},\ldots,\frac{1}{2^{k-2}},\frac{1}{2^{k-1}},\frac{1}{2^{k-1}} \right)
\end{align*}
can be modelled in the sense of Corollary \ref{gencor}. In particular, when $k=3,$ it reduces to \cite[Th\'eor\`eme 1.2]{de2016processus}.
\end{example}

Unsurprisingly, we expect that Theorem \ref{thmgeneral} should also hold for polynomials or permutations. Specifically, in the realm of permutations, the counterpart to multiplicative functions is the generalized Ewens measure (see \cite{elboim2022multiplicative}). Detailed proofs will be provided in the author's doctoral thesis. 

\section*{Acknowledgements}
The author is grateful to Andrew Granville and Dimitris Koukoulopoulos for their suggestions and encouragement. He would also like to thank Sary Drappeau for pointing out relevant papers, and the anonymous referee for helpful comments and corrections.

\printbibliography

\end{document}